\documentclass[a4paper,12pt]{article}

\usepackage{latexsym}
\usepackage{mathrsfs}
\usepackage{amsfonts,amsmath,amssymb}
\usepackage{indentfirst}
\usepackage{subeqnarray}
\usepackage{verbatim}
\usepackage{epstopdf}
\usepackage{algorithm}
\usepackage{algorithmic}
\usepackage{graphicx}
\usepackage{subfigure}
\usepackage{color}

\newtheorem{theorem}{Theorem}[section]
\newtheorem{lemma}[theorem]{Lemma}
\newtheorem{assumption}[theorem]{Assumption}

\newtheorem{definition}[theorem]{Definition}

\newtheorem{remark}[theorem]{Remark}

\numberwithin{equation}{section}
\newenvironment{proof}[1][Proof]{\textbf{#1.} }
{\ \rule{0.75em}{0.75em}\smallskip}

\begin{document}

\begin{center}
\Large\bf A projected gradient method for

$\alpha\ell_{1}-\beta\ell_{2}$ sparsity regularization
\end{center}

\begin{center}
Liang Ding\footnote{Department of Mathematics, Northeast Forestry University, Harbin 150040, China;
e-mail: {\tt dl@nefu.edu.cn}. The work of this author was supported by the Fundamental Research Funds
for the Central Universities (no.\ 2572018BC02), Heilongjiang Postdoctoral Research Developmental Fund
(no.\ LBH-Q16008), the National Nature Science Foundation of China (no.\ 41304093).}\quad and \quad
Weimin Han\footnote{Department of Mathematics, University of Iowa, Iowa City, IA 52242, USA;
e-mail: {\tt weimin-han@uiowa.edu}.}
\end{center}

\smallskip
\begin{quote}
{\bf Abstract.} The non-convex
$\alpha\|\cdot\|_{\ell_1}-\beta\| \cdot\|_{\ell_2}$ $(\alpha\ge\beta\geq0)$ regularization has attracted attention in the field of sparse recovery. One way to obtain a minimizer of this regularization is the ST-($\alpha\ell_1-\beta\ell_2$) algorithm which is similar to the classical iterative soft thresholding algorithm (ISTA).
It is known that ISTA converges quite slowly, and a faster alternative to ISTA is the
projected gradient (PG) method. However, the conventional PG method is limited to the classical $\ell_1$ sparsity regularization.
In this paper, we present two accelerated alternatives to the ST-($\alpha\ell_1-\beta\ell_2$) algorithm by extending the PG method to the non-convex $\alpha\ell_1-\beta\ell_2$
sparsity regularization. Moreover, we discuss a strategy to determine the radius $R$ of the $\ell_1$-ball constraint by Morozov's discrepancy principle. Numerical results are reported to illustrate the efficiency of the proposed approach.
\end{quote}


\section{Introduction}
In this paper, we are interested in solving an ill-posed operator equation of the form
\begin{equation}\label{equ1_1}
Ax=y,
\end{equation}
where $x$ is sparse, $A:\ell_2\rightarrow Y$ is a linear and bounded operator mapping between the $\ell_2$ space and a
 Banach space $Y$ with norms $\|\cdot\|_{\ell_2}$ and $\|\cdot\|_Y$, respectively. In practice, the right-hand side $y$ is known
only approximately with an error up to a level $\delta\geq 0$. Therefore, we assume that we know $\delta\geq 0$ and $y^{\delta}\in Y$ with $\|y^{\delta}-y\|_{Y}\leq \delta$.
The most commonly adopted technique to solve problem \eqref{equ1_1} is the $\ell_p$-norm sparsity regularization with $1\leq p<2$, see the monographs \cite{F2010,SGGHL2009} and the special
issues \cite{BB18,DDD16,JM12,JMS17} for many developments on regularizing properties and minimization schemes.
Since the $\ell_p$-norm regularization with $1\leq p<2$ does not always provide the ‘sparsest’ solution, the non-convex $\ell_p$-norm sparsity regularization with $0\leq p<1$ was proposed as alternatives. For the $\ell_0$ sparsity regularization, see \cite{BD08,BD09,BL08,FR08} for the iterative hard thresholding algorithm. We refer the reader to \cite{HSY15,LPZ19,MLP13} for some other types of alternatives to the $\ell_0$-norm.

The investigation of the non-convex
$\alpha\|\cdot\|_{\ell_1}-\beta\| \cdot\|_{\ell_2}$ $(\alpha\ge\beta\geq0)$ regularization has attracted attention in the field of sparse recovery over
the last five years, see \cite{DH19,LCGK20,LY18,YSX17,YLHX15} and references therein.
In \cite{DH19}, we investigated the well-posedness and convergence rate of the non-convex
$\alpha\|\cdot\|_{\ell_1}-\beta\| \cdot\|_{\ell_2}$ $(\alpha\ge\beta\geq0)$ sparsity regularization of the form
\begin{equation}\label{equ1_2}
\min\mathcal{J}_{\alpha,\beta}^{\delta}(x)=\frac{1}{q}\| Ax-y^{\delta}\|_Y^{q}+\mathcal{R}_{\alpha,\beta}(x)
\end{equation}
 in the $\ell_2$ space, where
\[ \mathcal{R}_{\alpha,\beta}(x):=\alpha\|x\|_{\ell_1}-\beta\| x\|_{\ell_2}, \quad \alpha\ge\beta\geq 0, \ q\geq 1. \]
Denoting $\eta=\beta/\alpha$, we can equivalently express the function $\mathcal{J}_{\alpha,\beta}^{\delta}(x)$ in (\ref{equ1_2}) as
\[ \frac{1}{q}\| Ax-y^{\delta}\|_Y^{q}+\alpha\mathcal{R}_{\eta}(x), \]
where
\[ \mathcal{R}_{\eta}(x):=\|x\|_{\ell_1}-\eta\| x\|_{\ell_2},\quad \alpha>0,\ 1\ge\eta\geq0. \]
For the particular case $q=2$, we provided an ST-($\alpha\ell_{1}-\beta\ell_{2}$) algorithm of the form
\begin{equation}\label{equ1_3}
z^{k}=\mathbb{S}_{\frac{\alpha}{\lambda}}\left(\left(\frac{\beta}{ \lambda\|x^k\|_{\ell_2}}+1\right)x^k-\frac{1}{\lambda}A^{*}(Ax^k-y^{\delta})\right), \quad
 x^{k+1}=x^{k}+s^k(z^k-x^k)
\end{equation}
for \eqref{equ1_2}, where $s^k$ is the step size and $\lambda>0$. Obviously, the ST-($\alpha\ell_{1}-\beta\ell_{2}$) algorithm is
similar to the classical ISTA when the step size $s^k=1$.  In \cite{DDD04}, an ISTA of the form
\begin{equation}\label{equ1_4}
x^{k+1}=\mathbb{S}_{\alpha}\left(x^k-A^{*}(Ax^k-y^{\delta})\right)
\end{equation}
was first proposed to solve the classical $\ell_1$ sparsity regularization of the form
\begin{equation}\label{equ1_5}
\min\mathcal{J}_{\alpha}^{\delta}(x)=\frac{1}{2}\| Ax-y^{\delta}\|_Y^{2}+\alpha\|x\|_{\ell_1}.
\end{equation}

As an alternative of the $\ell_p$-norm with $0\leq p<1$, the function $ \alpha\|\cdot\|_{\ell_1}-\beta\| \cdot\|_{\ell_2}$
$(\alpha\ge\beta\geq0)$ has the desired property that it is a good approximation of a multiple of the
$\ell_0$-norm. The function has a simpler structure than the $\ell_0$-norm from the perspective of computation.
The ST-($\alpha\ell_{1}-\beta\ell_{2}$) algorithm can easily be implemented, see \cite{DH19,GWC18,YLHX15} for
several other algorithms for $ \|\cdot\|_{\ell_1}-\| \cdot\|_{\ell_2}$ sparsity regularization.
However, the ST-($\alpha\ell_{1}-\beta\ell_{2}$) algorithm, in general, can be arbitrarily slow and it is computationally intensive.
So it is desirable to develop accelerated versions of the ST-($\alpha\ell_{1}-\beta\ell_{2}$) algorithm, especially for large-scale ill-posed inverse problems.

\subsection{Some accelerated algorithms for ISTA}

Searching for accelerated algorithms of the ISTA has become popular and some faster algorithms have been proposed.
In \cite{BF08,DFL08,FNW07,WNF09}, several accelerated
projected gradient methods have been provided. A comparison among several accelerated algorithms is provided in \cite{LBDZZ09}, including ``fast ISTA'' (\cite{BT09}). Applying a smoothing technique from Nesterov (\cite{N05}), a fast and accurate first-order method is proposed for solving large-scale compressed sensing problems (\cite{BBC11}). In \cite{DC15}, a simple heuristic adaptive restart technique is introduced, which can dramatically improve the convergence rate of accelerated gradient schemes. In \cite{CD15}, convergence of the iterates of the “Fast Iterative Shrinkage/Thresholding Algorithm” is established. In \cite{OBGXY05}, a new iterative regularization procedure for inverse problems based on the use of Bregman distances is studied. Numerical results show that the proposed method gives significant improvement over the standard method. An explicit algorithm based on a primal-dual approach for the minimization of an $\ell_1$-penalized least-squares function, with a non-separable $\ell_1$ term, is proposed in \cite{LV11}. An iteratively reweighted least squares algorithm and the corresponding convergence analysis for the regularization of linear inverse problems with sparsity constraints are investigated in \cite{FPRW16}. For a projected gradient method of nonlinear ill-posed problems, see \cite{TB10}.

Unfortunately, the algorithms stated above are only limited to the classical $\ell_1$-norm sparsity regularization.
Though there is great potential for accelerated algorithms in sparsity regularization with a non-convex penalty term,
to the best of our knowledge, little work can be found in the literature.  In \cite{RZ12},
the authors treat the problem of minimizing a general continuously differentiable
function subject to $\|x\|_0\leq s$, where $s>0$ is an integer, and $\|x\|_0$ is the $\ell_0$-norm of $x$, which counts the number of nonzero components in $x$. In this paper,
we extend the projected gradient method to the non-convex $\alpha\ell_1-\beta\ell_2$ sparsity regularization. There are two reasons why we choose PG method. First, its formulation is simple and it can easily be implemented. Another reason is that it converges quite fast. So it is adequate for solving large-scale ill-posed problems.

The PG method was introduced in \cite{DFL08} to accelerate the ISTA.  It is shown that the ISTA
converges initially relatively fast, then it overshoots the $\ell_1$-norm penalty, and it takes many steps to
re-correct back. It means that the algorithm generates a path $\{x_n\mid n\in \mathbb{N}\}$
that is initially fully contained in the $\ell_1$-ball $B_R:=\{x\in\ell_2 \mid\|x\|_{\ell_1}\leq R\}$. Then it gets out of the ball to slowly inch back to it in the limit.
To avoid this long “external” detour, the authors of \cite{DFL08} proposed an accelerated algorithm by substituting the soft thresholding operation $\mathbb{S}_{\alpha}$ by the projection $\mathbb{P}_{R}$ which is defined in Definition \ref{def2_4}. This leads to a projected gradient method of the form
\begin{equation}\label{equ1_6}
x^{k+1}=\mathbb{P}_{R}\left(x^k-\gamma^kA^{*}(Ax^k-y^{\delta})\right).
\end{equation}

\subsection{Contribution and organization}\label{ssec1_2}

Since the ST-($\alpha\ell_1-\beta\ell_2$) algorithm \eqref{equ1_3} is similar to ISTA \eqref{equ1_4}, inspired by \cite{DFL08}, we propose two accelerated alternatives to \eqref{equ1_3} by extending the PG method to solve \eqref{equ1_2}.

The first accelerated algorithm is based on the generalized conditional gradient method (GCGM). In \cite{DH19}, baed on GCGM, we proposed
the ST-($\alpha\ell_1-\beta\ell_2$) algorithm where the crucial issue is to determine $z^k$ by the optimization problem of the form
\begin{equation}\label{equ1_10}
 \begin{array}{llc}
\displaystyle
\min\limits_{z} \langle A^{*}(Ax^k-y^{\delta})-\lambda x^k-\frac{\beta x^k}{\|x^k\|_{\ell_2}},z\rangle+\frac{\lambda}{2}\|z\|_{\ell_2}^2+\alpha\|z\|_{\ell_1}.
 \end{array}
 \end{equation}
In this paper, we show that the problem \eqref{equ1_10} can be solved by a PG method of the form
\begin{equation}\label{equ1_11}
z^{k}=\mathbb{P}_{R}\left(x^k+\frac{\beta x^{k}}{\lambda\|x^{k}\|_{\ell_2}}-\frac{1}{\lambda}A^{*}(Ax^k-y^{\delta})\right).
\end{equation}
With $z^k$ at our disposal, we compute $x^{k+1}$ by $x^{k+1}=x^{k}+s^k(z^k-x^k)$, where $s^k$ is the step size.

Theoretically, the radius $R$ of $\ell_1$-ball should be chosen by $R=\|x_{\alpha,\beta}^{\delta}\|_{\ell_1}$ (\cite{DFL08}), where $x_{\alpha,\beta}^{\delta}$ is a minimizer of \eqref{equ1_2}.
However, in general, one can not obtain the value of $\|x_{\alpha,\beta}^{\delta}\|_{\ell_1}$ before starting the iteration \eqref{equ1_11}. In this paper, we utilize Morozov's discrepancy principle to determine $R$. This method only requires knowledge of the noise level $\delta$ and the observed data $y^{\delta}$. Moreover, we investigate the well-posedness of \eqref{equ1_2} under Morozov's discrepancy principle.

The second accelerated algorithm is based on the surrogate function approach. We investigate this algorithm in the finite dimensional space $\mathbb{R}^n$. For the case $q=2$, \eqref{equ1_2} takes the form
\begin{equation}\label{equ1_2add}
\min\mathcal{J}_{\alpha,\beta}^{\delta}(x)=\frac{1}{2}\| Ax-y^{\delta}\|_{\ell_2}^{2}+\alpha\|x\|_{\ell_1}-\beta\| x\|_{\ell_2},
\end{equation}
where $A:\mathbb{R}^n\rightarrow \mathbb{R}^m$ is a linear and bounded operator mapping between the $\mathbb{R}^n$ and $\mathbb{R}^m$ space
  with $\|\cdot\|_{\ell_2}$ norms.
In the following, we remove the $\ell_1$ constraint in \eqref{equ1_2add} and to consider a constrained optimization problem for a certain radius $R$ of $\ell_1$-ball constraint.
So, in analogy to the techniques about projection in \cite{DFL08,T96},
a natural strategy is to consider the constrained optimization problem of the form
\begin{equation}\label{equ1_7}
\min\frac{1}{2}\| Ax-y^{\delta}\|_{\ell_{2}}^{2}\quad
{\rm subject~to}~x\in B'_R:=\{x\in\mathbb{R}^n \mid\|x\|_{\ell_1}-\eta\|x\|_{\ell_2}\leq R\},\quad 1\ge\eta\geq 0.
\end{equation}
However, since $B'_R$ is non-convex, it is challenge to analyze and solve this constrained optimization problem.
To utilize the theory of convex constraints, we remove the $\ell_1$ constraint in \eqref{equ1_2add} and to consider instead the following optimization problem of the form
\begin{equation}\label{equ1_8}
\min\mathcal{D}_{\beta}^{\delta}(x)=\frac{1}{2}\| Ax-y^{\delta}\|_{\ell_2}^{2}-\beta\| x\|_{\ell_2}\quad
{\rm subject~to}~ x\in B_R:=\{x\in\mathbb{R}^n \mid\|x\|_{\ell_1}\leq R\}
\end{equation}
for a suitable $R$. We propose a projected gradient method of the form
\begin{equation}\label{equ1_9}
x^{k+1}=\mathbb{P}_{R}\left(x^k+\frac{\beta x^{k+1}}{\lambda\|x^{k+1}\|_{\ell_2}}-\frac{1}{\lambda}A^{*}(Ax^k-y^{\delta})\right)
\end{equation}
for \eqref{equ1_8}, where $\lambda>0$ satisfies some conditions, see Assumption \ref{assumption3_8}.

An outline of the rest of this paper is as follows. In the next section we introduce the notation and review results of the Tikhonov regularization and the PG method. In Section \ref{sec3}, we investigate an accelerated algorithm via GCGM. Furthermore, we give a strategy to determine the radius $R$ of $\ell_1$-ball constraint. In Section \ref{sec5}, we propose another accelerated algorithm via the surrogate function approach. Finally, we present results from numerical experiments on compressive sensing and image deblurring problems in Section \ref{sec6}.


\section{Preliminaries}\label{sec2}

Before starting the discussion on the accelerated algorithms, we briefly introduce some notation and results of the Tikhonov regularization and the PG method.  Let
\begin{equation}\label{equ2_1}
x^{\delta}_{\alpha,\beta}= \arg\min\limits_x\{\frac{1}{2}\|Ax-y^{\delta}\|_Y^{2}
+\mathcal{R}_{\alpha,\beta}(x)\}
\end{equation}
be a minimizer of the regularization function $\displaystyle \mathcal{J}_{\alpha,\beta}^{\delta}(x)$
in (\ref{equ1_2}) with $q=2$ for every $\alpha\ge\beta\geq0 $. We denote by $\mathcal{L}^{\delta}_{\alpha,\beta}$ the set
of all minimizers $x^{\delta}_{\alpha,\beta}$, and by $x^{\delta}_{R,\beta}$ a solution of \eqref{equ1_8}.
We use the following definition of $\mathcal{R}_{\eta}$-minimum solution (\cite{DH19}).

\begin{definition}\label{def2_1}
An element $x^{\dagger}\in \ell_2$ is called an $\mathcal R_{\eta}$-minimum solution of the
linear problem $Ax=y$ if
\[\displaystyle Ax^{\dagger}=y~~and~~\displaystyle \mathcal R_{\eta}(x^{\dagger})
=\min\limits_x\{\mathcal R_{\eta}(x)\mid Ax=y\}.\]
\end{definition}

We recall the definition of sparsity (\cite{DDD04}).

\begin{definition}\label{Defadd}
An element $x\in \ell_2$ is called sparse if $\mathrm{supp}(x):=\{i\in\mathbb{N}\mid x_{i}\neq0\}$ is finite,
where $x_i$ is the $i^{\rm th}$ component of $x$. $\|x\|_{0}:=\mathrm{supp}(x)$ is the cardinality
of $\mathrm{supp}(x)$. If $\|x\|_{0}=s$ for some $s\in \mathbb{N}$, then $x\in \ell_2$ is called $s$-sparse.
\end{definition}

\begin{definition}\label{def2_2}{\rm  (Morozov's discrepancy principle)}
For $1<\tau_1\leq\tau_2$, we choose $\alpha=\alpha(\delta, y^{\delta})>0$ such that
\begin{equation}\label{equ2_2}
\tau_1\delta\leq \|A x_{\alpha,\beta}^{\delta}-y^{\delta}\|_Y\leq \tau_2 \delta
\end{equation}
holds for some $x_{\alpha,\beta}^{\delta}$.
\end{definition}

Next we recall definitions of the soft thresholding and the projection operators (\cite{BF08,DDD04}).

\begin{definition}\label{def2_3}
For a given $\alpha>0$, the soft thresholding operator is defined as
\[\mathbb{S}_{\alpha}(x)=\sum\limits_{i}S_{\alpha}(x_i)e_i,\]
where $e_i=(\underbrace{0,\cdots,0,1}_i,0,\cdots)$, $x_i$ is the $i^{\rm th}$ component of $x$ and
\begin{align*}
\displaystyle S_{\alpha}(t)= \left\{\begin{array}{ll}
\displaystyle t-\alpha~~~~{\rm if}~~~t\geq\alpha, \\[2mm]
\displaystyle 0~~~~~~~~~{\rm if}~~~|t|<\alpha, \\[2mm]
\displaystyle t+\alpha~~~~{\rm if}~~~t \leq-\alpha.
\end{array}
\right. 	
\end{align*}
\end{definition}

\begin{definition}\label{def2_4}
The projection onto the $\ell_1$-ball is defined by
\[\mathbb{P}_R(\hat{x}):= \{\arg\min\limits_x \|x-\hat{x}\|_{\ell_2}\ {\rm subject\ to}\ \|x\|_{\ell_1}\leq R\},\]
which gives the projection of an element $\hat{x}$ onto the $\ell_1$-norm ball with radius $R>0$.
\end{definition}

Then we review two results from \cite{DFL08} on relations between the soft thresholding operator and the projection operator. For relations between the parameters $\alpha$ and $R$, see \cite[Fig.\ 2]{DFL08}.
\begin{lemma}\label{lemma2_5}
For some countable index set $\Lambda$, denote $\ell_p=\ell_p(\Lambda)$, $1\leq p<\infty$. For any fixed $a\in \ell_2(\Lambda)$ and for $\alpha>0$, $\|\mathbb{S}_{\alpha}(a)\|_{\ell_1}$ is a piecewise linear,
continuous, decreasing function of $\alpha$.  Moreover, if $a\in \ell_1(\Lambda)$ then $\|\mathbb{S}_{0}(a)\|_{\ell_1}=\|a\|_{\ell_1}$ and
$\|\mathbb{S}_{0}(a)\|_{\ell_1}=0$ for $\alpha\geq \max_i |a_i|$.
\end{lemma}

\begin{lemma}\label{lemma2_6}
If $\|a\|_{\ell_1}> R$, then the $\ell_2$ projection of $a$ on the $\ell_1$-ball with radius $R$ is given by $\mathbb{P}_{R}(a)=\mathbb{S}_{\alpha}(a)$, where $\alpha$
(depending on $a$ and $R$) is chosen such that $\|\mathbb{S}_{\alpha}(a)\|_{\ell_1}=R$. If $\|a\|_{\ell_1}\leq R$ then $\mathbb{P}_{R}(a)=\mathbb{S}_{0}(a)=a$.
\end{lemma}

Finally, recall the following properties of $\mathbb{P}_R(x)$ (\cite{DFL08}).

\begin{lemma}\label{lemma2_7}
Let $H$ be a Hilbert space with the inner product $\langle\cdot,\cdot\rangle$ and norm $\|\cdot\|_H$.  For any $x\in H$, $\mathbb{P}_R(x)$ is characterized as the unique vector in $H$ such that
\[\langle w-\mathbb{P}_R(x), x-\mathbb{P}_R(x)\rangle\leq 0 \quad \forall\, w\in {B}_R.\]
Moreover, the projection $\mathbb{P}_R$ is non-expansive:
\[\|\mathbb{P}_R(x')-\mathbb{P}_R(x'')\|_H\leq \|x'-x''\|_H\quad \forall\, x',x''\in H.\]
\end{lemma}



\section{The projected gradient method via GCGM}\label{sec3}

In \cite{DH19}, we proposed an ST-($\alpha\ell_{1}-\beta\ell_{2}$) algorithm for \eqref{equ1_2} based on GCGM. We rewrite $\mathcal{J}_{\alpha,\beta}^{\delta}(x)$ in (\ref{equ1_2}) as
\[  \mathcal{J}_{\alpha,\beta}^{\delta}(x)=F(x)+\Phi(x),  \]
where
\begin{align*}
F(x) & =\frac{1}{2}\| Ax-y^{\delta}\|_Y^2-\Theta(x), \\
\Phi(x) & =\Theta(x)+\alpha\|x\|_{\ell_1}-\beta\| x\|_{\ell_2}, \\
\Theta(x) & =\frac{\lambda}{2}\|x\|_{\ell_2}^2+\beta\|x\|_{\ell_2}, \quad \lambda>0.
\end{align*}
The ST-($\alpha\ell_{1}-\beta\ell_{2}$) algorithm is stated in the form of Algorithm \ref{alg1}. Convergence of Algorithm \ref{alg1} is given in Theorem \ref{theorem5_1}; see \cite[Theorem 3.5]{DH19} for its proof.

\begin{algorithm} 
\caption{ST-$({\alpha \ell_1-\beta \ell_2})$ algorithm for problem (\ref{equ1_2}) with $q=2$}
\label{alg1}
\begin{algorithmic}
\STATE{Set $k=0$, $x^0\in \ell_2$ such that $\Phi(x^0)<+\infty$,}
\STATE{for $k$ = 0, 1, 2, $\cdots$, do}
\STATE{~~~~if $x^k=0$ then}
\STATE{~~~~~~$\displaystyle x^{k+1}=\arg\min\frac{1}{2}\|Ax-y^{\delta}\|_Y^{2}+\alpha\|x\|_{\ell_1}$}
\STATE{~~~~else}
\STATE{~~~~~~determine a descent direction $z^k$ as a solution of
\[\min\limits_{z} \langle A^{*}(Ax^k-y^{\delta})-\lambda x^k-\frac{\beta x^k}{\|x^k\|_{\ell_2}},z\rangle+\frac{\lambda}{2}\|z\|_{\ell_2}^2+\alpha\|z\|_{\ell_1}\]}
\STATE{~~~~~~determine a step size $s^k$ as a solution of \[\min\limits_{s \in [0,1]} F(x^k+s(z^k-x^k))+\Phi(x^k+s(z^k-x^k))\]}
\STATE{~~~~~~$x^{k+1}=x^k+s^k(z^k-x^k)$}
\STATE{~~~~end if}
\STATE{~~~~$k=k+1$
\STATE{end for}}
\end{algorithmic}
\end{algorithm}

\begin{theorem}\label{theorem5_1}
Let $\{x^k\}$ denote the sequence generated by Algorithm \ref{alg1}.
Then $\{x^k\}$ contains a convergent subsequence and every convergent subsequence
of $\{x^k\}$ converges to a stationary point of the function $\mathcal{J}_{\alpha,\beta}^{\delta}(x)$ in \eqref{equ1_2} with $q=2$.
\end{theorem}

A crucial step in Algorithm \ref{alg1} is the determination of $z^k$ as a solution of
\begin{equation}\label{equ5_1}
\min\mathcal{C}_{\alpha,\beta,\lambda}^{\delta}(z,x^k)=\langle A^{*}(Ax^k-y^{\delta})-\lambda x^k-\frac{\beta x^k}{\|x^k\|_{\ell_2}},z\rangle+\frac{\lambda}{2}\|z\|_{\ell_2}^2+\alpha\|z\|_{\ell_1}.
\end{equation}
In \cite{DH19}, we solve \eqref{equ5_1} by
\begin{equation}\label{equ5_2}
 z^{k}=\mathbb{S}_{\alpha/\lambda}\left(\left(\frac{\beta}{\lambda \|x^k\|_{\ell_2}}+1\right)x^k-\frac{1}{\lambda}A^{*}(Ax^k-y^{\delta})\right).
\end{equation}
However, \eqref{equ5_2} is known to converge quite slowly.  To accelerate the ST-$({\alpha \ell_1-\beta \ell_2})$ algorithm, we transform \eqref{equ5_1} to an $\ell_1$-ball constraint optimization problem of the form
\begin{equation}\label{equ5_3}
\displaystyle \left\{\begin{array}{ll}
\displaystyle \min\mathcal{D}_{\beta,\lambda}^{\delta}(z,x^k)=\langle A^{*}(Ax^k-y^{\delta})-\lambda x^k-\frac{\beta x^k}{\|x^k\|_{\ell_2}},z\rangle+\frac{\lambda}{2}\|z\|_{\ell_2}^2,\quad \beta\geq 0, \\[2mm]
 {\rm {subject~to}}~ \ell_1~{\rm ball}~B_R:=\{z\in \ell_2 \mid\|z\|_{\ell_1}\leq R\}.
\end{array}
\right. 	
\end{equation}
Since $\mathcal{C}_{\alpha,\beta,\lambda}^{\delta}(z,x^k)$, $\mathcal{D}_{\beta,\lambda}^{\delta}(z,x^k)$ and $B_R$ are convex  with respect to the variable $z$, problem \eqref{equ5_3} is equivalent to \eqref{equ5_1} for a certain $R$ (\cite[Theorem 27.4]{R1970}, \cite[Theorem 47.E]{Z1985}).
In Lemma \ref{lemma4_1}, we show that the problem \eqref{equ5_3} can be solved by a PG method of the form
\begin{equation}\label{equ5_3add}
z^{k}=\mathbb{P}_{R}\left(x^k+\frac{\beta x^{k}}{\lambda\|x^{k}\|_{\ell_2}}-\frac{1}{\lambda}A^{*}(Ax^k-y^{\delta})\right).
\end{equation}

\begin{lemma}\label{lemma4_1}
An element $\hat{z}\in B_R$ is a minimizer of \eqref{equ5_3} if and only if
\begin{equation}\label{equ5_4}
\hat{z}=\mathbb{P}_{R}\left(x^k+\frac{\beta x^{k}}{\lambda\|x^{k}\|_{\ell_2}}-\frac{1}{\lambda}A^{*}(Ax^k-y^{\delta})\right)
\end{equation}
for any $\lambda>0$, which is equivalent to
\begin{equation}\label{equ5_5}
\left\langle x^k+\frac{\beta x^{k}}{\lambda\|x^{k}\|_{\ell_2}}-\frac{1}{\lambda}A^{*}(Ax^k-y^{\delta})-\hat{z}, z-\hat{z}\right\rangle\leq 0\quad\forall\,z\in B_R.
 \end{equation}
\end{lemma}
\begin{proof}
Note that $\hat{z}\in B_R$ is a solution of \eqref{equ5_3} if and only if for any $z\in B_R$, the
function $f(t)=\mathcal{D}_{\beta,\lambda}^{\delta}((1-t)\hat{z}+tz,x^k)$ of $t\in [0,1]$ attains its minimum at $t=0$.  Since $f(t)$ is
quadratic and convex, a necessary and sufficient condition for $f(0)=\min_{0\le t\le 1}f(t)$ is $f^\prime(0+)\ge 0$.  Easily,
\[ f^\prime(0+)=\langle A^{*}(Ax^k-y^{\delta})-\lambda\,x^k-\frac{\beta x^k}{\|x^k\|_{\ell_2}}+\lambda\,\hat{z},z-\hat{z}\rangle, \]
and $f^\prime(0+)\ge 0$ is equivalent to \eqref{equ5_5}.
\end{proof}

The PG algorithm for \eqref{equ1_2} based on GCGM is stated in the form of Algorithm \ref{alg2}.

\begin{algorithm} 
\caption{PG algorithm for problem (\ref{equ1_2}) based on GCGM}
\label{alg2}
\begin{algorithmic}
\STATE{Choose $x^0\in \ell_2$, $\beta=O(\delta)$, $\Phi(x^0)<+\infty$,}
\STATE{~~~~for $k$ = 0, 1, 2, $\cdots$, do}
\STATE{~~~~~~~~if $x^k=0$ then}
\STATE{~~~~~~~~~~~$x^{k+1}=\arg\min\frac{1}{2}\|Ax-y^{\delta}\|_Y^{2}+\alpha\|x\|_{\ell_1}$}
\STATE{~~~~~~~~else}
\STATE{~~~~~~~~~~determine $z^k$ by \[z^k=\mathbb{P}_{R}\left(x^k+\frac{\beta x^{k}}{\lambda\|x^{k}\|_{\ell_2}}-\frac{1}{\lambda}A^{*}(Ax^k-y^{\delta})\right)\]}
\STATE{~~~~~~~~~~determine a step size $s^k$ as a solution of \[\min\limits_{s \in [0,1]} F(x^k+s(z^k-x^k))+\Phi(x^k+s(z^k-x^k))\]}
\STATE{~~~~~~~~~~$x^{k+1}=x^k+s^k(z^k-x^k)$}
\STATE{~~~~~~~~end if}
\STATE{~~~~$k=k+1$
\STATE{~~~~end for}}
\end{algorithmic}
\end{algorithm}



\subsection{Determination of the radius $R$}\label{subsec4}

From the previous discussion, we know that \eqref{equ5_1} is equivalent to \eqref{equ5_3} for a certain $R$. Before starting iteration \eqref{equ5_3add}, we need to choose an appropriate value of $R$ which is crucial
for the computation, especially in practical application. In this section, we give a strategy to determine the radius $R$ of the $\ell_1$-ball constraint by Morozov's discrepancy principle.

By Lemma \ref{lemma2_6}, for a given $\alpha$ in \eqref{equ5_1}, $R$ in \eqref{equ5_3} should be chosen such that $R=\|x_{\alpha,\beta}^{\delta}\|_{\ell_1}$. However, one does not know the value of $\|x_{\alpha,\beta}^{\delta}\|_{\ell_1}$ before starting \eqref{equ5_3add}. Of course, we can find an approximation of $x_{\alpha,\beta}^{\delta}$
by the ST-($\alpha\ell_1-\beta\ell_2$) algorithm \eqref{equ1_3}. Nevertheless, this implies that an additional soft thresholding iteration \eqref{equ1_3} is needed in Algorithm \ref{alg2}. Then the resulting algorithm is no longer an accelerated one.

So a crucial issue is how to check whether a value of $R$ is appropriate for \eqref{equ5_3}.
Recall that there exists a regularization parameter $\alpha$ depending on $R$ such that \eqref{equ5_1} is equivalent to \eqref{equ5_3}.
So to determine an appropriate $R$, we need to check whether the corresponding regularization parameter $\alpha$ is appropriate.
One criterion is to check whether $\delta^2=O(\alpha)$. If $\delta^2=O(\alpha)$, then $x_{\alpha,\beta}^{\delta}$ is a regularized solution (\cite[Theorem 2.13]{DH19}). However, by Lemmas \ref{lemma2_5} and \ref{lemma2_6}, we only know that $\alpha$ is a piecewise linear,
continuous, decreasing function of $R$ (see \cite[Fig.\ 2]{DFL08}), and there is no explicit formula relating $\alpha$ and $R$. We can not determine the
value of $\alpha$ from the value of $R$ directly. So we can not ensure whether the $R$ is appropriate.

\begin{algorithm}
\caption{The PG algorithm for problem (\ref{equ1_2}) based on GCGM}
\label{alg3}
\begin{algorithmic}
\STATE{Choose $x^0\in \ell_2$, $R_0\in \mathbb{R}^+$, $\beta=O(\delta)$, $\Phi(x^0)<+\infty$,}
\STATE{for $j$ = 0, 1, 2, $\cdots$ do}
\STATE{~~~~for $k$ = 0, 1, 2, $\cdots$, do}
\STATE{~~~~~~~~if $x^k=0$ then}
\STATE{~~~~~~~~~~~$x^{k+1}=\arg\min\frac{1}{2}\|Ax-y^{\delta}\|_Y^{2}+\alpha\|x\|_{\ell_1}$}
\STATE{~~~~~~~~else}
\STATE{~~~~~~~~~~determine $z^k$ by \[z^k=\mathbb{P}_{R_j}\left(x^k+\frac{\beta x^{k}}{\lambda\|x^{k}\|_{\ell_2}}-\frac{1}{\lambda}A^{*}(Ax^k-y^{\delta})\right)\]}
\STATE{~~~~~~~~~~determine a step size $s^k$ as a solution of \[\min\limits_{s \in [0,1]} F(x^k+s(z^k-x^k))+\Phi(x^k+s(z^k-x^k))\]}
\STATE{~~~~~~~~~~$x^{k+1}=x^k+s^k(z^k-x^k)$}
\STATE{~~~~~~~~end if}
\STATE{~~~~$k=k+1$
\STATE{~~~~end for}}
\STATE{~~~~~~~~if (2.2) is satisfied, set $R_{j+1}=R_j+c$, $c>1$}
\STATE{~~~~~~~~~~~otherwise stop iteration}
\STATE{~~~~~~~~end if}
\STATE{~~~~$j=j+1$}
\STATE{end for}
\end{algorithmic}
\end{algorithm}

Another criterion is Morozov's discrepancy principle. For any given $R$, we should check whether the regularization parameter $\alpha$ satisfies Morozov's discrepancy principle \eqref{equ2_2}, i.e.
\[ \tau_1\delta\leq \|Ax_{\alpha,\beta}^{\delta}-y^{\delta}\|_Y\leq \tau_2\delta,\quad 1<\tau_1\leq\tau_2.\]
For any fixed $R$, we need to compute $x_{\alpha,\beta}^{\delta}$ by Algorithm \ref{alg2} where $z^k$ is determined by the PG method \eqref{equ5_3add}. Subsequently, we check whether $x_{\alpha,\beta}^{\delta}$ satisfies \eqref{equ2_2}. For this strategy, we only need to know the observed data $y^{\delta}$ and the noise level $\delta$. By Lemma \ref{lemmaadd}, the discrepancy $\|Ax_{\alpha,\beta}^{\delta}-y^{\delta}\|_Y$ is an increasing function of $\alpha$. A commonly adopted technique is to try $\displaystyle \alpha_j=\alpha\,2^{-j}$, $j=1, 2, \cdots$. With $j$ increasing, one calculates $x_{\alpha,\beta}^\delta$ until one finds $\alpha=\inf\{\alpha>0 \mid \tau_1\delta\leq \|Ax_{\alpha,\beta}^{\delta}-y^{\delta}\|_Y\leq \tau_2\delta\}$ (\cite{TA1977}). Since $\alpha$ is a decreasing function of $R$, the discrepancy $\|Ax_{\alpha,\beta}^{\delta}-y^{\delta}\|_Y$ is a decreasing function of $R$, see Lemma \ref{lemma2_5} and Fig.\ \ref{fig1}. Hence $R:=\sup\{R>0 \mid \tau_1\delta\leq \|Ax_{\alpha,\beta}^{\delta}-y^{\delta}\|_Y\leq \tau_2\delta\}$ is a reasonable choice. We begin with a small $R$ such that $x_{\alpha,\beta}^{\delta}$ satisfies Morozov's discrepancy principle \eqref{equ2_2}. Subsequently, we increase the value of $R$ to $R+c$, $c\in \mathbb{Z}^+$, until $x_{\alpha,\beta}^{\delta}$ fails to satisfy Morozov's discrepancy principle. Then we can find a maximal $R$ which satisfies Morozov's discrepancy principle \eqref{equ2_2}. Of course, we can also begin with a large $R$ and gradually reduce the value of $R$ until $R$ satisfies Morozov's discrepancy principle \eqref{equ2_2}. Under Morozov's discrepancy principle, the PG algorithm for \eqref{equ1_2} based on GCGM is stated in the form of Algorithm \ref{alg3}.

 A natural question is whether \eqref{equ1_2} combined with Morozov's discrepancy principle is a regularization method. As we know, Tikhonv type functions combined with Morozov's discrepancy principle is a regularization method. However, this result is usually shown only when the regularized term is convex (\cite{AR10,B09,R02,S93,TA1977,TLY1998}). If the regularized term is non-convex, some results can be found in \cite{DH19,WLMC13} where Morozov's discrepancy principle is applied to derive the convergence rate. However, these results are obtained under additional source conditions on the true solution $x^{\dag}$. To the best of our knowledge, no results are available on whether Morozov's discrepancy principle combined with \eqref{equ1_2} is a regularization method. In this paper, we prove that if the non-convex regularized term satisfies some properties, e.g. coercivity, weakly lower semi-continuity and Radon-Riesz property, the well-posedness of the regularization still holds.

\subsection{Well-posedness of regularization}

In this section, we discuss the well-posedness of \eqref{equ1_2} under Morozov's discrepancy principle. First, we show that there exists at least one regularization parameter $\alpha$ in \eqref{equ1_2} such that Morozov's discrepancy principle \eqref{equ2_2} holds.
We recall some properties of $\mathcal{R}_{\alpha,\beta}(x)$ (\cite{DH19}), needed in analyzing the well-posedness of \eqref{equ1_2}.

\begin{lemma}\label{lemma2_711} If $\alpha>\beta$, the function $\mathcal{R}_{\alpha,\beta}(x)$ in \eqref{equ1_2} has the following properties:

{\rm(i)  (Coercivity)} For $x\in \ell_2$, $\|x\|_{\ell_2}\rightarrow \infty$ implies $\mathcal{R}_{\alpha,\beta}(x)\rightarrow \infty$.

{\rm(ii)  (Weak lower semi-continuity)} If $x_n\rightharpoonup x$ in $\ell_2$ and $\{\mathcal{R}_{\alpha,\beta}(x_n)\}$
is bounded, then
\[\liminf_n\mathcal{R}_{\alpha,\beta}(x_n)\geq\mathcal{R}_{\alpha,\beta}(x).\]

{\rm (iii) (Radon-Riesz property)} If $x_n\rightharpoonup x$ in $\ell_2$ and $\mathcal{R}_{\alpha,\beta}(x_n)\rightarrow
\mathcal{R}_{\alpha,\beta}(x)$, then $\|x_n-x\|_{\ell_2}\rightarrow 0$.
\end{lemma}

\begin{definition}\label{def2_8}
For fixed $\delta$ and $\eta\in[0,1]$, define
\begin{align*}
F(x_{\alpha,\beta}^{\delta}) & =\frac{1}{2}\|Ax_{\alpha,\beta}^{\delta}-y^{\delta}\|_Y^{2},\\
\mathcal{R}_{\eta}(x_{\alpha,\beta}^{\delta}) & =\|x_{\alpha,\beta}^{\delta}\|_{\ell_1}-\eta\| x_{\alpha,\beta}^{\delta}\|_{\ell_2},\\
m(\alpha)& =\mathcal{J}_{\alpha,\beta}^{\delta}(x_{\alpha,\beta}^{\delta})=\min \mathcal{J}_{\alpha,\beta}^{\delta}(x),
\end{align*}
where $\alpha\in (0, \infty)$ and $\beta=\alpha\eta$.
\end{definition}


In the following we give some properties of $m(\alpha)$, $F(x_{\alpha,\beta}^{\delta})$ and $\mathcal{R}_{\eta}(x_{\alpha,\beta}^{\delta})$ in Lemmas \ref{lemmaadd} and \ref{lemma2_10}. Since $\mathcal{R}_{\eta}(x_{\alpha,\beta}^{\delta})$ is weakly lower semi-continuous, the proofs are similar to that in \cite[Section 2.6]{TA1977}.  Note that $\eta\in [0,1]$ is fixed, and for given $\alpha_1,\alpha_2\in(0,\infty)$, we write $\beta_1=\alpha_1\eta$ and $\beta_2=\alpha_2\eta$.

\begin{lemma}\label{lemmaadd}
The function $m(\alpha)$ is continuous and non-increasing, i.e., $\alpha_1>\alpha_2$ implies $m(\alpha_1)\leq m(\alpha_2)$.  Moreover,
for $\alpha_1>\alpha_2$,
\begin{align*}
\sup_{x_{\alpha_1,\beta_1}^{\delta}\in \mathcal{L}^{\delta}_{\alpha_1,\beta_1}}F(x_{\alpha_1,\beta_1}^{\delta})
 &\leq  \inf_{x_{\alpha_2,\beta_2}^{\delta}\in \mathcal{L}^{\delta}_{\alpha_2,\beta_2}}F(x_{\alpha_2,\beta_2}^{\delta}) , \\
\sup_{x_{\alpha_1,\beta_1}^{\delta}\in \mathcal{L}^{\delta}_{\alpha_1,\beta_1}}\mathcal{R}_{\eta}(x_{\alpha_1,\beta_1}^{\delta})
& \geq  \inf_{x_{\alpha_2,\beta_2}^{\delta}\in \mathcal{L}^{\delta}_{\alpha_2,\beta_2}}\mathcal{R}_{\eta}(x_{\alpha_2,\beta_2}^{\delta}).
\end{align*}
\end{lemma}

\begin{lemma}\label{lemma2_10}
For each $\bar{\alpha}>0$ there exist $x', x''\in \mathcal{L}^{\delta}_{\bar{\alpha},\bar{\beta}}$ such that
\[\lim\limits_{\alpha\rightarrow {\bar{\alpha}}^{-}}\left(\sup_{x_{\alpha,\beta}^{\delta}\in \mathcal{L}^{\delta}_{\alpha,\beta}}F(x_{\alpha,\beta}^{\delta})\right)=F(x')=\inf\limits_{x\in \mathcal{L}^{\delta}_{\bar{\alpha},\bar{\beta}}}F(x)\quad {\rm and} \quad \lim\limits_{\alpha\rightarrow \bar{\alpha}^{+}}\left(\inf_{x_{\alpha,\beta}^{\delta}\in \mathcal{L}^{\delta}_{\alpha,\beta}}F(x_{\alpha,\beta}^{\delta})\right)=F(x'')=\sup\limits_{x\in \mathcal{L}^{\delta}_{\bar{\alpha},\bar{\beta}}}F(x). \]
\end{lemma}

In the following we provide an existence result on the regularization parameter $\alpha$. The proof is along the line of Morozov's discrepancy principle for nonlinear ill-posed problems (\cite{AR10,R02}).

\begin{lemma}\label{lemma2_11}
Assume $0<c_2\delta<\|y^{\delta}\|_Y$.  Then there exist $\alpha_1, \alpha_2\in \mathbb{R}^+$ such that
\[\sup_{x_{\alpha_1,\beta_1}^{\delta}\in \mathcal{L}^{\delta}_{\alpha_1,\beta_1}}F(x_{\alpha_1,\beta_1}^{\delta})<\tau_1\delta\leq\tau_2\delta<\inf_{x_{\alpha_2,\beta_2}^{\delta}\in \mathcal{L}^{\delta}_{\alpha_2,\beta_2}}F(x_{\alpha_2,\beta_2}^{\delta}).\]
\end{lemma}
\begin{proof}
First, let $\alpha_n\rightarrow 0$ and consider a sequence of corresponding minimizers $x_n:=x_{\alpha_n, \beta_n}^{\delta}\in \mathcal{L}^{\delta_n}_{\alpha_n,\beta_n}$. By the definition of $x_{\alpha,\beta}^{\delta}$
and $x^{\dag}$, we have
\[F(x_n)^q\leq m(\alpha_n)\leq \mathcal{J}_{\alpha_n}(x^{\dag})\leq\delta^{q}+\alpha_n \mathcal{R}_{\eta}(x^{\dag})\rightarrow \delta^q<\tau_1^q\delta^q.\]
This implies that there exists a small enough $\alpha_1$ such that $\sup_{x_{\alpha_1,\beta_1}^{\delta}\in \mathcal{L}^{\delta}_{\alpha_1,\beta_1}}F(x_{\alpha_1,\beta_1}^{\delta})<\tau_1\delta$.

Next, let $\alpha_n\rightarrow \infty$.  Then
\begin{equation}\label{equ2_3}
\mathcal{R}_{\eta}(x_n)\leq \frac{1}{\alpha_n}m(\alpha_n)\leq \frac{1}{\alpha_n}\|A0-y^{\delta}\|_Y\rightarrow 0.
\end{equation}
From the definition of $\mathcal{R}_{\eta}(x)$,
\begin{equation}\label{equ2_4}
\mathcal{R}_{\eta}(x)=(1-\eta)\,\|x\|_{\ell_1}+\eta\,(\|x\|_{\ell_1}-\|x\|_{\ell_2}).
\end{equation}
Then a combination of \eqref{equ2_3} and \eqref{equ2_4} implies that $\{\|x_n\|_{\ell_2}\}$ is bounded. Consequently, $\{x_n\}$ has a convergent subsequence, again denoted by $\{x_n\}$, such that $x_n\rightharpoonup x^*$ for some $x^*\in\ell_2$.
By Lemma \ref{lemma2_711} (ii), it follows from \eqref{equ2_3} that
\[  0\leq \mathcal{R}_{\eta}(x^*)\leq \liminf \mathcal{R}_{\eta}(x_n)=\lim \mathcal{R}_{\eta}(x_n)=0.\]
By \eqref{equ2_4}, this implies $x^*=0$. Since $x_n\rightharpoonup 0$ and $\mathcal{R}_{\eta}(x_n)\rightarrow \mathcal{R}_{\eta}(0)$, Lemma \ref{lemma2_711} (iii) implies that $x_n\rightarrow 0$. Then
\[\|Ax_n-y^{\delta}\|_Y\rightarrow \|A0-y^{\delta}\|_Y=\|y^{\delta}\|_Y>c_2\delta.\]
This implies that there exists a large enough $\alpha_2$ such that $\inf_{x_{\alpha_2,\beta_2}^{\delta}\in \mathcal{L}^{\delta}_{\alpha_2,\beta_2}}F(x_{\alpha_2,\beta_2}^{\delta})>\tau_2\delta$.
\end{proof}

Note that we require $\|y^{\delta}\|_Y>c_2\delta$ in Lemma \ref{lemma2_11}, which is a reasonable assumption. Indeed, in applications, it is almost impossible to recover a solution from observed data of a size in the same order as the noise.

We state an existence result on the regularized parameter, similar to Theorems 3.10 in \cite{AR10}.
The proof makes use of the properties stated in Lemmas \ref{lemma2_10} and \ref{lemma2_11}.

\begin{theorem}\label{theorem2_12}
Assume $\|y^{\delta}\|_Y>c_2\delta>0$ and there is no $\alpha>0$ with minimizers $x', x''\in \mathcal{L}^{\delta}_{\alpha,\beta}$ such that
\[\|Ax'-y^{\delta}\|_Y<\tau_1\delta\leq \tau_2\delta<\|Ax''-y^{\delta}\|_Y.\]
Then there exist $\alpha=\alpha(\delta, y^{\delta})>0$ and $x_{\alpha,\beta}^{\delta}\in \mathcal{L}^{\delta}_{\alpha,\beta}$
such that \eqref{equ2_2} holds.
\end{theorem}

Next, we give the convergence of \eqref{equ1_2} under Morozov's discrepancy principle.

\begin{theorem}\label{theorem2_13} {\rm (Convergence)}
Let $\displaystyle x_{\alpha_{n},\beta_{n}}^{\delta_{n}}$ be a minimizer of
$\mathcal{J}_{\alpha_n,\beta_n}^{\delta_n}(x)$ defined by \eqref{equ2_1} with the data $y^{\delta_n}$
satisfying $\| y-y^{\delta_n}\|\leq\delta_n$, where $\delta_n\rightarrow 0$ if $n\rightarrow +\infty$
and $y^{\delta_n}$ belongs to the range of $A$.
Let $\alpha_n$ be determined by Morozov's discrepancy principle,
\[ \tau_1\delta_n\leq \|A(x_{\alpha_{n},\beta_{n}}^{\delta_{n}})-y^{\delta_n}\|_Y\leq \tau_2 \delta_n,
\quad 1<\tau_1\leq\tau_2.\]
Moreover, assume that $\displaystyle\eta=\lim\limits_{n\to\infty}\eta_n\in [0,1)$ exists, where
$\eta_n=\beta_n/\alpha_n$. Then there exists a subsequence of $\{x_{\alpha_{n},\beta_{n}}^{\delta_{n}}\}$, denoted by $\{x_{\alpha_{n_k},\beta_{n_k}}^{\delta_{n_k}}\}$, that
converges to an $\mathcal{R}_{\eta}$-minimizing solution $x^{\dag}$ in $\ell_2$.
If, in addition, the $\mathcal{R}_{\eta}$-minimizing solution $x^{\dag}$ is unique, then
\[\lim\limits_{n \rightarrow +\infty} \| x_{\alpha_{n},\beta_n}^{\delta_{n}}-x^{\dag}\|_{\ell_2}=0.\]
\end{theorem}
\begin{proof}
Denote $y_{n}:=y^{\delta_{n}}$, $x_{n}:=x_{\alpha_{n},\beta_{n}}^{\delta_{n}}$, $\eta_{n}:=\eta^{\delta_{n}}$.
By the definition of $x_{n}$, we obtain
\begin{align}
\displaystyle \frac{1}{q}\| Ax_{n}-y_{n}\|_Y^{q} +\alpha_{n}\| x_{n}\|_{\ell_1}-\beta_{n}\| x_{n}\|_{\ell_2}
&\leq\frac{1}{q}\|Ax^{\dag}-y_n\|_Y^q+\alpha_{n}\|x^{\dag}\|_{\ell_1}-\beta_{n}\|x^{\dag}\|_{\ell_2}\nonumber\\
&\displaystyle\leq\frac{1}{q}\delta_n^q+\alpha_n\|x^{\dag}\|_{\ell_1}-\beta_{n}\|x^{\dag}\|_{\ell_2}.
\label{equ2_5}
\end{align}
Since $\tau_1\delta_n\leq \|Ax_{n}-y_n\|_Y$, it follows from \eqref{equ2_5} that
\[\alpha_{n}\| x_{n}\|_{\ell_1}-\beta_{n}\| x_{n}\|_{\ell_2}\leq \alpha_{n}\|x^{\dag}\|_{\ell_1}-\beta_{n}\|x^{\dag}\|_{\ell_2}.\]
Then we have
\begin{equation}\label{equ2_6}
\displaystyle \limsup\limits_{n \to +\infty}\left(\| x_n \|_{\ell_1}-\eta_n\| x_n \|_{\ell_2}\right)
\leq\| x^{\dag} \|_{\ell_1}-\eta\| x^{\dag} \|_{\ell_2}.
 \end{equation}
Since $\| x_n \|_{\ell_2}$ is bounded, there exist an $x^{*}\in\ell_2$
and a subsequence of $\{x_{n_k}\}$ such that $x_{n_k}\rightharpoonup x^{*}$ in $\ell_2$.
By Morozov's discrepancy principle, we obtain
\[\| Ax_{n_k}-y\|_Y\leq\| Ax_{n_k}-y_{n_k}\|_Y+\| y-y_{n_k}\|_Y\leq (\tau_2+1)\delta_{n_k}.\]
Therefore, weak lower semicontinuity of the norm gives
\begin{equation}\label{equ2_7}
\|Ax^*-y\|\leq \liminf\limits_{k\rightarrow \infty}\|Ax_{n_k}-y\|_Y=0.
\end{equation}
Meanwhile, by \eqref{equ2_6} and Lemma \ref{lemma2_711} (ii), we have
\begin{align}\label{equ2_8}
\displaystyle \| x^{*} \|_{\ell_1}-\eta\| x^{*} \|_{\ell_2}
&\displaystyle\leq\liminf\limits_{k}(\|x_{n_k} \|_{\ell_1}-\eta_{n_k}\| x_{n_k} \|_{\ell_2})\leq\limsup\limits_{k}(\|x_{n_k} \|_{\ell_1}-\eta_{n_k}\| x_{n_k} \|_{\ell_2})\nonumber\\
&\displaystyle\leq\limsup\limits_{n}(\|x_{n} \|_{\ell_1}-\eta_{n}\| x_{n} \|_{\ell_2})\leq\| x^{\dag} \|_{\ell_1}-\eta\| x^{\dag} \|_{\ell_2}.
 \end{align}
By the definition of $x^{\dag}$, a combination of \eqref{equ2_7} and \eqref{equ2_8} implies that $x^{*}$ is an $\mathcal{R}_{\eta}$-minimizing solution.
 Hence, $\lim\limits_{k\rightarrow \infty}\mathcal{R}_{\eta}(x_{n_k})\rightarrow \mathcal{R}_{\eta}(x^*)$. By Lemma \ref{lemma2_711} (iii), we have $x_{n_k}\rightarrow x^*$. If the $\mathcal{R}_{\eta}$-minimizing solution is unique, then $x^{*}=x^{\dag}$. This implies that, for every subsequence $\{x_{n_k}\}$, $x_{n_k}$ converges to $x^{\dag}$, then we have $\lim\limits_{n \rightarrow +\infty} \| x_{n}-x^{\dag} \|_{\ell_2}=0$.
\end{proof}

The numerical experiments in \cite{DH19} show that
we can obtain satisfactory results even when $\alpha=\beta$. Indeed, $\mathcal{R}_{\alpha,\beta}(x)$
behaves more and more like the $\ell_0$-norm as $\beta/\alpha\rightarrow 1$. Nevertheless, note that if $\alpha=\beta$, $\mathcal{R}_{\alpha,\alpha}(x)$ fails to satisfy the coercivity and the Radon-Riesz property, and we can not ensure the convergence in $\ell_2$-norm.
Without the Radon-Riesz property, we may expect to have only weak convergence for the regularized solution. If we assume the operator
$A$ is coercive in $\ell_2$, i.e.\ $\|x\|_{\ell_2}\rightarrow \infty$ implies $\|Ax\|_Y\rightarrow \infty$, then the proof of the weak convergence is similar to that of Theorem \ref{theorem2_13}.



\section{The projected gradient method via the surrogate function approach}\label{sec5}

In this section, we propose another projected gradient algorithm for \eqref{equ1_2} in the finite dimensional space $\mathbb{R}^n$
based on the surrogate function approach. By the discussion in Subsection \ref{ssec1_2}, we consider the optimization problem \eqref{equ1_8}.
The following result provides a first order optimality condition for \eqref{equ1_8}.

\begin{lemma}\label{lemma3_2}
Let $0\not=\hat{w}\in \mathbb{R}^n$ be a minimizer of \eqref{equ1_8}. Then
\begin{equation}\label{equ3_2}
\mathbb{P}_{R}\left(\hat{w}+\frac{\beta \hat{w}}{\lambda\|\hat{w}\|_{\ell_2}}-\frac{1}{\lambda} A^{*}(A\hat{w}-y^{\delta})\right)=\hat{w}
 \end{equation}
for any $\lambda>0$, equivalently,
\begin{equation}\label{equ3_3}
\left\langle\frac{\beta \hat{w}}{\|\hat{w}\|_{\ell_2}}-A^{*}(A\hat{w}-y^{\delta}), w-\hat{w}\right\rangle\leq 0,
 \end{equation}
for all $w\in B_R$.
\end{lemma}
\begin{proof}
By the definition of $\hat{w}$, for any $w\in B_R$, the function
\[ f(t)=\frac{1}{2}\| A((1-t)\hat{w}+tw)-y^{\delta}\|_{\ell_2}^{2}-\beta\| (1-t)\hat{w}+tw\|_{\ell_2},\quad t\in [0,1] \]
has its minimum at $t=0$.  Thus,
\[ f^\prime(0+)=\langle A\hat{w}-y^{\delta},A(w-\hat{w})\rangle-\beta\|\hat{w}\|_{\ell_2}^{-1}\langle \hat{w},w-\hat{w}\rangle\ge 0,\]
i.e., \eqref{equ3_3} holds.
\end{proof}

Due to the non-convexity of $\mathcal{D}_{\beta}^{\delta}(x)$, \eqref{equ3_3} is only a necessary condition of \eqref{equ1_8}.

\begin{lemma}\label{lemma3_3}
For any fixed $\beta\geq 0$, define
\begin{equation}\label{equ3_8}
\Phi_{\lambda}(w, x):=\frac{1}{2}\|Aw-y^{\delta}\|_{\ell_2}^2-\beta\|w\|_{\ell_2}-\frac{1}{2}\|A(w-x)\|_{\ell_2}^2+\frac{\lambda}{2}\|w-x\|_{\ell_2}^2,\quad w, x\in B_R.
\end{equation}
Then for any fixed $x\in B_R$, there exists a minimizer $\hat{w}$ of $\Phi_{\lambda}(w, x)$ on $B_R$.
\end{lemma}
\begin{proof}
Being continuous, the function $\Phi_{\lambda}(\cdot, x)$ has a minimum on the compact set $B_R$.
\end{proof}

Note that a minimizer $\hat{w}$ of $\Phi_{\lambda}(w, x)$ depends on $x$ in $\Phi_{\lambda}(w, x)$.  For $w\not=0$, we denote
\[ a_{i,j}(w)=\frac{\partial^2 \|w\|_{\ell_2}}{\partial w_i\partial w_j},\quad 1\le i,j\le n.\]
Then,
\begin{equation}\label{equ3_9}
a_{i,j}(w)= \frac{\delta_{ij}}{\|w\|_{\ell_2}}-\frac{w_i w_j}{\|w\|_{\ell_2}^3},\quad 1\le i,j\le n.
\end{equation}

Since $\|w\|_{\ell_2}$ is convex, the matrix $(a_{i,j}(w))_{n\times n}$ is positive semi-definite. Thus, $ {\rm eig}(w)\geq 0$, where $ {\rm eig} (w)$ denotes the eigenvalues of $(a_{i,j}(w))_{n\times n}$. Moreover, $\max\{{\rm eig} (w)\}$ is an increasing function of $\|w\|_{\ell_2}$.

\begin{lemma}\label{lemma3_5}
Let $\hat{w}$ be a minimizer of $\Phi_{\lambda}(w, x)$. For a fixed $\beta\geq 0$ and a fixed nonzero $x\in B_R$, there exists $\lambda>0$ such that $\lambda> \max\{{\rm eig} (\hat{w})\}$.
\end{lemma}
\begin{proof}
As $\lambda\rightarrow +\infty$ in \eqref{equ3_8}, all minimizers $\hat{w}$ of $\Phi_{\lambda}(w, x)$ converge to $x$. Then $ {\rm eig}(\hat{w})\rightarrow {\rm eig}(x)$. Since $0\not=x\in B_R$ is fixed, there exists a large enough $\lambda$ such that $\lambda\geq \max_n\{ {\rm eig}(\hat{w})\}$.
\end{proof}

\begin{lemma}\label{lemma3_6}
For a nonzero minimizer $\hat{w}$ of $\Phi_{\lambda}(w, x)$ and a fixed $\beta\geq 0$, if $\lambda\geq \beta\max\{ {\rm eig}(\hat{w})\}$, then $\Phi_{\lambda}(w, x)$ is locally convex.
\end{lemma}
\begin{proof}
By the definition of $\Phi_{\lambda}(w, x)$,
\[  \frac{\partial^2 \Phi_{\lambda}(w, x)}{\partial w_i\partial w_j}=\lambda\,\delta_{ij}-\beta\,a_{i,j}(w),\quad 1\le i,j\le n. \]
By the assumption $\lambda\geq \beta\max\{ {\rm eig}(\hat{w})\}$, the Hessian matrix
$(\frac{\partial^2 \Phi_{\lambda}(w, x)}{\partial w_i\partial w_j}\big|_{w=\hat{w}})$ is positive semi-definite. This proves the lemma.
\end{proof}

In Lemma \ref{lemma3_6}, we assume $\lambda\geq \beta\max\{ {\rm eig}(\hat{w})\}$. This condition is weaker than
$\lambda\geq \max\{ {\rm eig}(\hat{w})\}$. In general, the regularization parameter $\alpha\ll 1$ in the Tihkonov regularization.
Since $\beta=\alpha\eta$ and $0\le\eta\leq 1$, we also have $\beta\ll 1$.

\begin{lemma}\label{lemma3_7}
Let $0\not=\hat{w}\in B_R$ and $\lambda\geq \beta\max\{ {\rm eig}(\hat{w})\}$.  Then $\hat{w}$ is a minimizer of $\Phi_{\lambda}(w, x)$ on $B_R$ if and only if
\begin{equation}\label{equ1_80}
\hat{w}=\mathbb{P}_{R}\left(x+\frac{\beta \hat{w}}{\lambda\|\hat{w}\|_{\ell_2}}-\frac{1}{\lambda} A^{*}(Ax-y^{\delta})\right).
\end{equation}
\end{lemma}
\begin{proof}
By the definition of $\hat{w}$, for any $w\in B_R$, the function
\begin{align*}
f(t)&=\frac{1}{2}\| A((1-t)\hat{w}+tw)-y^{\delta}\|_{\ell_2}^{2}-\beta\| (1-t)\hat{w}+tw\|_{\ell_2}\\
&\quad{} -\frac{1}{2}\|A((1-t)\hat{w}+tw-x)\|_{\ell_2}^2+\frac{\lambda}{2}\|(1-t)\hat{w}+tw-x\|_{\ell_2}^2,\quad t\in [0,1]
\end{align*}
has its minimum at $t=0$.  Thus,
\begin{align*}
f^\prime(0+)& =\langle A\hat{w}-y^{\delta},A(w-\hat{w})\rangle-\beta\|\hat{w}\|_{\ell_2}^{-1}\langle \hat{w},w-\hat{w}\rangle-\langle A\hat{w}-Ax,A(w-\hat{w})\rangle+\lambda\langle \hat{w}-x,w-\hat{w}\rangle\\
& \ge 0,
\end{align*}
i.e.,
\[
\left\langle \frac{1}{\lambda} A^*(Ax-y)+\hat{w}-x-\frac{\beta}{\lambda}\frac{\hat{w}}{\|\hat{w}\|_{\ell_2}}, w-\hat{w}\right\rangle\geq 0.
\]
By Lemma \ref{lemma2_7}, this implies \eqref{equ1_80}.

On the other hand, let now $\hat{w}\in B_R$ be such that \eqref{equ1_80} holds.  By Lemma \ref{lemma2_7}, we have
\[
\left\langle x+\frac{\beta\hat{w}}{\lambda\|\hat{w}\|_{\ell_2}}-\frac{1}{\lambda} A^*(Ax-y)-\hat{w}, w-\hat{w}\right\rangle\leq 0.
\]
Define
\begin{equation}\label{equ1_81}
J(w):=\Phi_{\lambda}(w, x)=\frac{1}{2}\|Aw-y\|_{\ell_2}^2-\beta\|w\|_{\ell_2}-\frac{1}{2}\|A(w-x)\|_{\ell_2}^2+\frac{\lambda}{2}\|w-x\|_{\ell_2}^2.
\end{equation}
If $w\neq 0$, we have
\begin{equation}\label{equ1_82}
J'(w)=A^*(Ax-y)+\lambda(w-x)-\beta\frac{w}{\|w\|_{\ell_2}}.
\end{equation}
By \eqref{equ1_82}, this implies that
\begin{equation}\label{equ1_86}
0\leq \langle J'(\hat{w}), w-\hat{w}\rangle=\lim\limits_{t\rightarrow 0^+}\frac{J(\hat{w}+t(w-\hat{w}))-J(\hat{w})}{t}.
\end{equation}
By assumption and Lemma \ref{lemma3_6}, $\Phi_{\lambda}(w, x)$ is locally convex at $\hat{w}$. This implies that
\begin{align*}
0&\leq \langle J'(\hat{w}), w-\hat{w}\rangle=\lim\limits_{t\rightarrow 0^+}\frac{J(\hat{w}+t(w-\hat{w}))-J(\hat{w})}{t}\nonumber
\\& \leq \lim\limits_{t\rightarrow 0^+}\frac{tJ(w)+(1-t)J(\hat{w})-J(\hat{w})}{t}=J(w)-J(\hat{w})
\end{align*}
for all $w\in B_R$. This proves the lemma.
\end{proof}

Denoting by $x^{k+1}$ the sequence generated by the formula
\begin{equation}
x^{k+1}=\mathbb{P}_{R}\left(x^k+\frac{\beta x^{k+1}}{\lambda\|x^{k+1}\|_{\ell_2}}-\frac{1}{\lambda}A^{*}(Ax^k-y^{\delta})\right).
\label{4.8a}
\end{equation}
The projected gradient algorithm based on the surrogate function is stated in the form of Algorithm \ref{alg4}.

\begin{algorithm}
\caption{PG algorithm for (\ref{equ1_8}) based on the surrogate function approach}
\label{alg4}
\begin{algorithmic}
\STATE{Choose $x^0\in \mathbb{R}^n$, $R_0\in \mathbb{R}^+$, $\beta=O(\delta)$ and $\lambda$ such that $\lambda>\beta\max\{\rm (eig(x^0), eig(x^{\dag})\}$}
\STATE{for $j$ = 0, 1, 2, $\cdots$ do}
\STATE{~~~~for $k$ = 0, 1, 2, $\cdots$ do}
\STATE{~~~~~~~~$x^{k+1}=\mathbb{P}_{R_j}\left(x^k+\frac{\beta x^{k+1}}{\lambda\|x^{k+1}\|_{\ell_2}}-\frac{1}{\lambda}A^{*}(Ax^k-y^{\delta})\right)$\quad (by fixed point iteration)}
\STATE{~~~~~~~~$k=k+1$}
\STATE{~~~~end for}
\STATE{~~~~~~~~if (2.2) is satisfied, set $R_{j+1}=R_j+c$, $c>1$}
\STATE{~~~~~~~~~~~otherwise stop iteration}
\STATE{~~~~~~~~end if}
\STATE{~~~~$j=j+1$}
\STATE{end for}
\end{algorithmic}
\end{algorithm}

To prove the convergence of Algorithm \ref{alg4}, we impose some restrictions on the operator $A$ and $\lambda$.

\begin{assumption}\label{assumption3_8}
Let $r:=\|A^*A\|_{{L(\mathbb{R}^n, \mathbb{R}^n)}}<1$. Assume that

${\rm {(A1)}}\quad \|Ax\|_{\ell_2}^2\leq \frac{\lambda\,r}{2}\|x\|_{\ell_2}^2$ for all $x\in \ell_2$

${\rm {(A2)}}\quad \lambda\geq \beta\max\{ {\rm eig}(x^k)\}$ for all $k$.
\end{assumption}

In Assumption \ref{assumption3_8}, we let $r:=\|A^*A\|_{{L(\mathbb{R}^n, \mathbb{R}^n)}}<1$. In the classical theory of
sparsity regularization, the value of $\|A_{m\times n}\|_2$ is assumed to be less than 1 (\cite{DDD04}), where $m$ denotes the number of rows in the operator $A$. This requirement is still needed in this paper. If $r>1$, we need to re-scale the original ill-posed problem by $A_{m\times n}x_n=y_m\rightarrow \left(\frac{1}{c}A_{m\times n}\right)x_n=\frac{1}{c}y_m$ so that $\frac{1}{c^2}\|A^*A\|_{{L(\mathbb{R}^n, \mathbb{R}^n)}}<1$, where $c>1$. If $r<1$, we let $\lambda>2$; then ${\rm {(A1)}}$ holds. As for ${\rm {(A2)}}$, it seems that we need to compute eigenvalues for every $(a_{ij}(x^k))_{n\times n}$. However, we can give an approximation for the eigenvalues of $(a_{ij}(x^k))_{n\times n}$. In this paper, we first estimate the value of $\|x^{\dag}\|_{\ell_2}$ and $\|x^{0}\|_{\ell_2}$, then we can give an approximation for the order of the maximal eigenvalues of $\|x^{\dag}\|_{\ell_2}$ and $\|x^{0}\|_{\ell_2}$. Subsequently, we choose $\lambda$ such that $\lambda$ is greater than the order of the maximal eigenvalues of $\|x^{\dag}\|_{\ell_2}$ and $\|x^{0}\|_{\ell_2}$. If the value of $\|x^{\dag}\|_{\ell_2}$ is too small, we can re-scale the original ill-posed problem by $A_{m\times n}x_n=y_m\rightarrow \left(\frac{1}{c}A_{m\times n}\right)(cx_n)=y_m$ to increase the value of $\|x^{\dag}\|_{\ell_2}$, where $c>1$. Meanwhile, this strategy can reduce the value of $\|A_{m\times n}\|_2$, see Section \ref{sec6} for details.

\begin{lemma}\label{lemma3_9}
Let Assumption \ref{assumption3_8} hold with $\{x^{k+1}\}$ generated by \eqref{4.8a}.  Then,
\[ \mathcal{D}_{\beta}^{\delta}(x^{k+1})\leq \mathcal{D}_{\beta}^{\delta}(x^{k})\]
and
\[\lim\limits_{k\rightarrow \infty}\|x^{k+1}-x^k\|_{\ell_2}=0.\]
\end{lemma}
\begin{proof}
By Lemma \ref{lemma3_7} and the definition of $x^{k+1}$, we see that $x^{k+1}$ is a minimizer of $\Phi_{\lambda}(w, x^k)$. Then we have
\begin{align*}
\mathcal{D}_{\beta}^{\delta}(x^{k+1})&\leq \mathcal{D}_{\beta}^{\delta}(x^{k+1})+\frac{2-r}{2r}\|A(x^{k+1}-x^k)\|_Y^2\\
& \leq\frac{1}{2}\| Ax^{k+1}-y\|_Y^{2}-\beta\| x^{k+1}\|_{\ell_2}+\frac{1}{r}\|A(x^{k+1}-x^{k})\|_Y^{2}-\frac{1}{2}\|A(x^{k+1}-x^{k})\|_Y^{2}\\
& \leq\frac{1}{2}\| Ax^{k+1}-y\|_Y^{2}-\beta\| x^{k+1}\|_{\ell_2}-\frac{1}{2}\|A(x^{k+1}-x^{k})\|_Y^{2}+\frac{\lambda}{2}\|x^{k+1}-x^{k}\|_{\ell_2}^{2}\\
& = \Phi_{\lambda}(x^{k+1}, x^{k})\leq \Phi_{\lambda}(x^{k}, x^{k})=\mathcal{D}_{\beta}^{\delta}(x^{k}).
\end{align*}
Furthermore,
\begin{align*}
\Phi_{\lambda}(x^{k+1}, x^{k})-\Phi_{\lambda}(x^{k+1}, x^{k+1})&=\frac{\lambda}{2}\|x^{k+1}-x^{k}\|_{\ell_2}^{2}-\frac{1}{2}\|A(x^{k+1}-x^{k})\|_Y^{2}\\
& \geq \frac{\lambda(2-r)}{4}\|x^{k+1}-x^{k}\|_{\ell_2}^{2}.
\end{align*}
This implies
\begin{align*}
\sum\limits_{k=0}^N\|x^{k+1}-x^{k}\|_{\ell_2}^{2}&\leq  \frac{4}{\lambda(2-r)}\sum\limits_{k=0}^N\left(\Phi_{\lambda}(x^{k+1}, x^{k})-\Phi_{\lambda}(x^{k+1}, x^{k+1}) \right)\\
&\leq \frac{4}{\lambda(2-r)}\sum\limits_{k=0}^N\left(\Phi_{\lambda}(x^{k}, x^{k})-\Phi_{\lambda}(x^{k+1}, x^{k+1}) \right)\\
&= \frac{4}{\lambda(2-r)}\left(\Phi_{\lambda}(x^{0}, x^{0})-\Phi_{\lambda}(x^{N+1}, x^{N+1}) \right)\\
&\leq  \frac{4}{\lambda(2-r)}(\Phi_{\lambda}(x^{0}, x^{0})+\beta R).
\end{align*}
Since $\sum\limits_{k=0}^N\|x^{k+1}-x^{k}\|_{\ell_2}^{2}$ is uniformly bounded with respect to $N$, the series $\sum_{k=0}^{\infty}\|x^{k+1}-x^{k}\|_{\ell_2}^{2}$
converges. This proves the lemma.
\end{proof}

\begin{remark}\label{remark3_10}
To prove the convergence, we need to analyze the relation between $x^k$ and 0. If $0=x^0=x^1$,
then we stop the iteration and 0 is the iterative solution. Otherwise, by Lemma \ref{lemma3_9},
$\mathcal{D}_{\beta}^{\delta}(x^{k})$ decreases, which implies that $x^k\neq0$ for $k\geq 1$. So in the following we let $x^k\neq0$ whenever $k\geq 1$.
\end{remark}

\begin{lemma}\label{lemma2_71}
Denote $\Psi(\hat{w}):=\mathbb{P}_{R}\left(x+\frac{\beta \hat{w}}{\lambda\|\hat{w}\|_{\ell_2}}-\frac{1}{\lambda} A^{*}(Ax-y^{\delta})\right)$. Then the fixed point iteration $\hat{w}^{l+1}=\Psi(\hat{w}^l)$ has a subsequence which converges to an element $\hat{w}$. If $\hat{w}\neq 0$, then $\hat{w}$ is a fixed point of $\Psi(\hat{w})$.
\end{lemma}
\begin{proof}
By Lemma \ref{lemma2_7}, $\mathbb{P}_{R}(x)$ is non-expansive,
\[
\|\Psi(\hat{w}_1)-\Psi(\hat{w}_2)\|_{\ell_2}\leq\left\|\frac{\beta \hat{w}_1}{\lambda\|\hat{w}_1\|_{\ell_2}}-\frac{\beta \hat{w}_2}{\lambda\|\hat{w}_2\|_{\ell_2}}\right\|_{\ell_2},
\]
which implies $\Psi(\hat{w})$ is continuous at any nonzero element $w$. Since $\{\hat{w}^l\}$ is bounded, it has a subsequence $\{\hat{w}^{l_k}\}$ which converges to an element $\hat{w}\in B_R$. Since $\hat{w}^{l_k+1}=\Psi(\hat{w}^{l_k})$,
\begin{equation}\label{equ1_88}
\lim\limits_k\hat{w}^{l_k+1}=\lim\limits_k\Psi(\hat{w}^{l_k}).
\end{equation}
If $\hat{w}\neq 0$, it follows from \eqref{equ1_88} that $\hat{w}=\Psi(\hat{w})$.
\end{proof}

Even though $\mathbb{P}_{R}(x)$ is non-expansive, the map $\Psi(\hat{w})$ is not necessarily non-expansive. So we only have the existence of a fixed point. We can not ensure uniqueness of the fixed point. Indeed, due to the non-convexity of $\Phi_{\lambda}(w, x)$ in \eqref{equ3_8}, the minimizer of
\eqref{equ3_8} may be non-unique. Nevertheless, the convergence still holds and the limit depends on the choice of the initial vector $x^0$.

\begin{theorem}\label{theorem3_12}{\rm(Convergence)}
Let $\{x^k\}$ be the sequence generated by
\[
x^{k+1}=\mathbb{P}_{R}\left(x^k+\frac{\beta x^{k+1}}{\lambda\|x^{k+1}\|_{\ell_2}}-\frac{1}{\lambda}A^{*}(Ax^k-y^{\delta})\right).
\]
Then $\{x^k\}$ has a subsequence which converges to a nonzero stationary point $x^*$ of \eqref{equ1_8}, i.e.\ $x^*$ satisfies
\[\left\langle\frac{\beta x^{*}}{\|x^{*}\|_{\ell_2}}-A^{*}(Ax^{*}-y^{\delta}), w-x^{*}\right\rangle\leq 0 \quad \forall\, w\in B_R.\]
\end{theorem}
\begin{proof}
Since $\{x^k\}\subset B_R$ is bounded, $\{x^k\}$ has a subsequence $\{x^{k_j}\}$ converging
to an element $x^*$ in $B_R$, i.e.\ $x^{k_j}\rightarrow x^*$ in $B_R$. Since $A$ is linear and bounded, $A(x^{k_j})\rightarrow A(x^*)$. By Lemma \ref{lemma2_7} and the definition of $x^{k+1}$, we see that, for all $w\in B_R$,
\[ \left\langle x^k+\frac{\beta x^{k+1}}{\lambda\|x^{k+1}\|_{\ell_2}}-\frac{1}{\lambda}A^{*}(Ax^k-y^{\delta})-x^{k+1}, w-x^{k+1}\right\rangle\leq 0.\]
This implies that
\begin{equation}\label{equ1_89}
\left\langle x^{k_j}+\frac{\beta x^{k_j+1}}{\lambda\|x^{k_j+1}\|_{\ell_2}}-\frac{1}{\lambda}A^{*}(Ax^{k_j}-y^{\delta})-x^{k_j+1}, w-x^{k_j+1}\right\rangle\leq 0.
\end{equation}
Taking the limit in \eqref{equ1_89} as $j\to\infty$, we have
\begin{equation}\label{equ3_20}
\lim\limits_{j\rightarrow \infty}\left\langle x^{k_j}+\frac{\beta x^{{k_j}+1}}{\lambda\|x^{{k_j}+1}\|_{\ell_2}}-\frac{1}{\lambda}A^{*}(Ax^{k_j}-y^{\delta})-x^{{k_j}+1}, w-x^{{k_j}+1}\right\rangle\leq 0.
\end{equation}
Since $\|x^{k_j}-x^{k_j+1}\|_{\ell_2}\rightarrow 0$ as $j\rightarrow \infty$ and $\{w-x^{k_j+1}\}$ is uniformly bounded, we have
\begin{equation}\label{equ3_21}
\lim\limits_{j\rightarrow \infty}|\langle x^{k_j}-x^{k_j+1}, w-x^{k_j+1}\rangle|=0.
\end{equation}
A combination of \eqref{equ3_20} and \eqref{equ3_21} shows that
\begin{equation}\label{equ3_22}
\lim\limits_{j\rightarrow \infty}\left\langle \frac{\beta x^{{k_j}+1}}{\lambda\|x^{{k_j}+1}\|_{\ell_2}}-\frac{1}{\lambda}A^{*}(Ax^{k_j}-y^{\delta}), w-x^{{k_j}+1}\right\rangle\leq 0.
\end{equation}
Since $x^{k_j}\rightarrow x^*$, it follows from \eqref{equ3_22} that
\[ \left\langle \frac{\beta x^{*}}{\|x^{*}\|_{\ell_2}}-A^{*}(Ax^{*}-y^{\delta}), w-x^{*}\right\rangle\leq 0.\]
by Lemma \ref{lemma3_2}, $x^*$ is a stationary point of $\mathcal{D}_{\beta}^{\delta}(x)$ on $B_R$.
\end{proof}

\begin{remark}\label{remark1}
In this section, we restrict the analysis of the projected algorithm based on the surrogate function approach in the finite dimensional space $\mathbb{R}^n$.
Actually, all results except Lemma \ref{lemma2_71} and Theorem \ref{theorem3_12} can be extended to $\ell_2$ space.
In Theorem \ref{theorem3_12}, if $\{x^k\}$ is defined in $\ell_2$ space, then
$\{x^k\}$ has a weak convergence subsequence $\{x^{k_j}\}\rightharpoonup x^*$. However, the challenge of the proof is that $x^{k_j}\rightharpoonup x^*$ can not ensure
$x^{{k_j}+1}/\|x^{{k_j}+1}\|_{\ell_2}\rightharpoonup x^{*}/\|x^{*}\|_{\ell_2}$.
For example,
let $x_n=\bar{x}+e_n$, where $e_n=(\underbrace{0,\cdots,0,1}_n,0,\cdots)$. Since $e_n\rightharpoonup 0$ in $\ell_2$, $x_n\rightharpoonup x$ in $\ell_2$. However,
$\|x_n\|_{\ell_2}\nrightarrow\|x\|_{\ell_2}$. Hence, $x_n/\|x_n\|_{\ell_2}$ does not converge to $x^{*}/\|x^{*}\|_{\ell_2}$.
If we impose an additional condition on $\{x_n\}$, e.g.\ $\|x_n\|_{\ell_2}\rightarrow\|x\|_{\ell_2}$, then we have $x_n/\|x_n\|_{\ell_2}\rightharpoonup \eta x^{*}/\|x^{*}\|_{\ell_2}$. However, this condition is too restrictive, since a combination of $\|x_n\|_{\ell_2}\rightarrow\|x\|_{\ell_2}$ and $x_{n}\rightharpoonup x^*$ in $\ell_2$ implies that $x_n\rightarrow x^*$. Moreover, the iterative algorithm in this paper has an implicit formulation, and we need to compute the iterative solution. However, in $\ell_2$ space, we do not know whether the operator $\Phi(\hat{w})$ is weak-strong continuity. So we can not ensure that the fixed point iteration is convergent.
\end{remark}



\section{Numerical experiments}\label{sec6}

In this section, we present results from two numerical experiments to demonstrate the efficiency of
the proposed algorithms. Comparisons between ST-($\alpha\ell_{1}-\beta\ell_{2}$) and the two projected gradient algorithms are provided.
For convenience, we write PG-GCGM algorithm to refer to the first projected gradient algorithm which is based on GCGM, and
PG-SF algorithm for the second projected gradient algorithm which is based on the surrogate function approach. The relative error (Rerror) is utilized to measure the performance of the reconstruction $x^{*}$:
\[  \mathrm{Rerror}:=\displaystyle\frac{\|x^{*}-x^{\dag}\|_{\ell_2}}{\|x^{\dag}\|_{\ell_2}},   \]
where $x^{\dag}$ is a true solution.

We utilize the algorithm in \cite[Section 4.2]{BF08} to compute the projection defined in Definition \ref{def2_4}. The MATLAB code oneProjector.m regarding the $\ell_1$-ball projection can be obtained at http://www.cs.ubc.ca/labs/scl/spgl1.
The first example deals with a well-conditioned compressive sensing problem. The second example deals
with an ill-conditioned image deblurring problem. All numerical experiments were tested in MATLAB R2010 on an i7-6500U 2.50GHz workstation with 8Gb RAM.

\subsection{Example 1: Compressive sensing}

In the first example, we test compressive sensing with the commonly used random Gaussian matrix. The compressive sensing problem
is defined as $A_{m\times n}x_n=y_m$, where $A_{m\times n}$ is a well conditioned random Gaussian matrix
by calling $\mathrm{A=randn(m,n)}$ in MATLAB. Exact data $y^{\dag}$ is generated by $y^{\dag}=Ax^{\dag}$.
The exact solution $x^{\dag}$ is an $s$-sparse signal supported on a random index set. White Gaussian noise
is added to the exact data $y^{\dag}$ by calling $\mathrm{y^{\delta}=awgn(Ax^{\dag}, \sigma})$ in MATLAB,
where $\sigma$ (measured in dB) measures the ratio between the true (noise free) data $y^{\dag}$ or $Ax^{\dag}$ and Gaussian noise. A larger value of $\sigma$ corresponds to a smaller value of the noise level $\delta$, where the noise level $\delta$ is defined by $\delta=\|y^{\delta}-y^{\dag}\|_{2}$. $x^{*}$ denotes the reconstruction computed by
the proposed algorithms.
For compressive sensing, if the value of $\|(A^*A)_{n\times n}\|_2$ is greater than 1, we rescale the matrix $A_{m\times n}$ by
$A_{m\times n}\rightarrow c*A_{m\times n}$, where $c<1$. Then the original compressive sensing problem $A_{m\times n}x_n=y_m$ can be rewritten as $(c*A_{m\times n})x_n=c*y_m$.
Note that the condition number does not change under the matrix rescaling.  To compare the performance of ST-($\alpha\ell_{1}-\beta\ell_{2}$) algorithm, PG-GCGM algorithm and PG-SF algorithm, we choose the same initial setting, i.e. $\lambda$, $\beta$ and the initial vector $x^0$. Moreover, for each fixed point iteration in PG-SF algorithm, we choose $x^0=\mathrm{ones}(n,1)$ as the initial vector.

We choose $n = 200$, $m = 0.4n$, $s = 0.2m$, then $\|x^{\dag}\|_{0}=16$. A noise $\delta$ is added to exact data $y^{\dag}$ by calling
$\mathrm{y^{\delta}=awgn(Ax^{\dag}, \sigma})$, where $\mathrm{\sigma}=50\mathrm{dB}$, $\delta$ is around 0.02. We let $\lambda = 1$, $\eta=1$, $\alpha=O(\delta)=0.2$, $\beta=\alpha\eta=0.2$ and the initial vector $x^0$ is generated by calling $x^0=0.01\mathrm{ones}(n,1)$. We utilize discrepancy principle \eqref{equ2_2} to determine the radius $R$ of the $\ell_1$-ball constraint such that $R=\sup\{R>0 \mid \delta\leq \|Ax^*-y^{\delta}\|_2\}$. It is shown that when a good estimate for the noise level $\delta$ is known, this method yields a good radius $R$. According to the priori information of $x^{\dag}$, we choose an initial value of $R$ and compute $x^*$.
 If $\delta< \|Ax^*-y^{\delta}\|_2$, we try $R_j=R+j$, $j=1, 2, \cdots$ until $\|Ax^*-y^{\delta}\|_2\leq \delta$. With $j$ increasing, we can find
 $R=\sup\{R>0 \mid \delta\leq \|Ax^{*}-y^{\delta}\|_2\}$. On the contrary, for any initial $R$, if $\|Ax^*-y^{\delta}\|_2\leq \delta$, we try $R_j=R-j$, $j=1, 2, \cdots$ until $\delta< \|Ax^*-y^{\delta}\|_2$. Fig.\ \ref{fig1} shows Morozov's discrepancy principle for determining the radius $R$. We see that the discrepancy $\|Ax^*-y^{\delta}\|_2$ is a decreasing function of the radius $R$. According the strategy stated above, $R$ should be chosen such that $R=\sup\{R>0 \mid \delta< \|Ax^*-y^{\delta}\|_2\}$. It is obvious that $R$ should be chosen as 16. Indeed, by ST-($\alpha\ell_{1}-\beta\ell_{2}$) algorithm, we can obtain $\|x^*\|_1=16.0153$. Thus the experimental results  confirm that the strategy proposed in this paper is feasible and they match the theoretical results stated in Subsection \ref{subsec4}, i.e.\ $R$ should be chosen by $R=\|x^*\|_1$.

\begin{figure}[tbhp]
\centering
\subfigure[PG-GCGM algorithm]{\includegraphics[width=80mm,height=60mm]{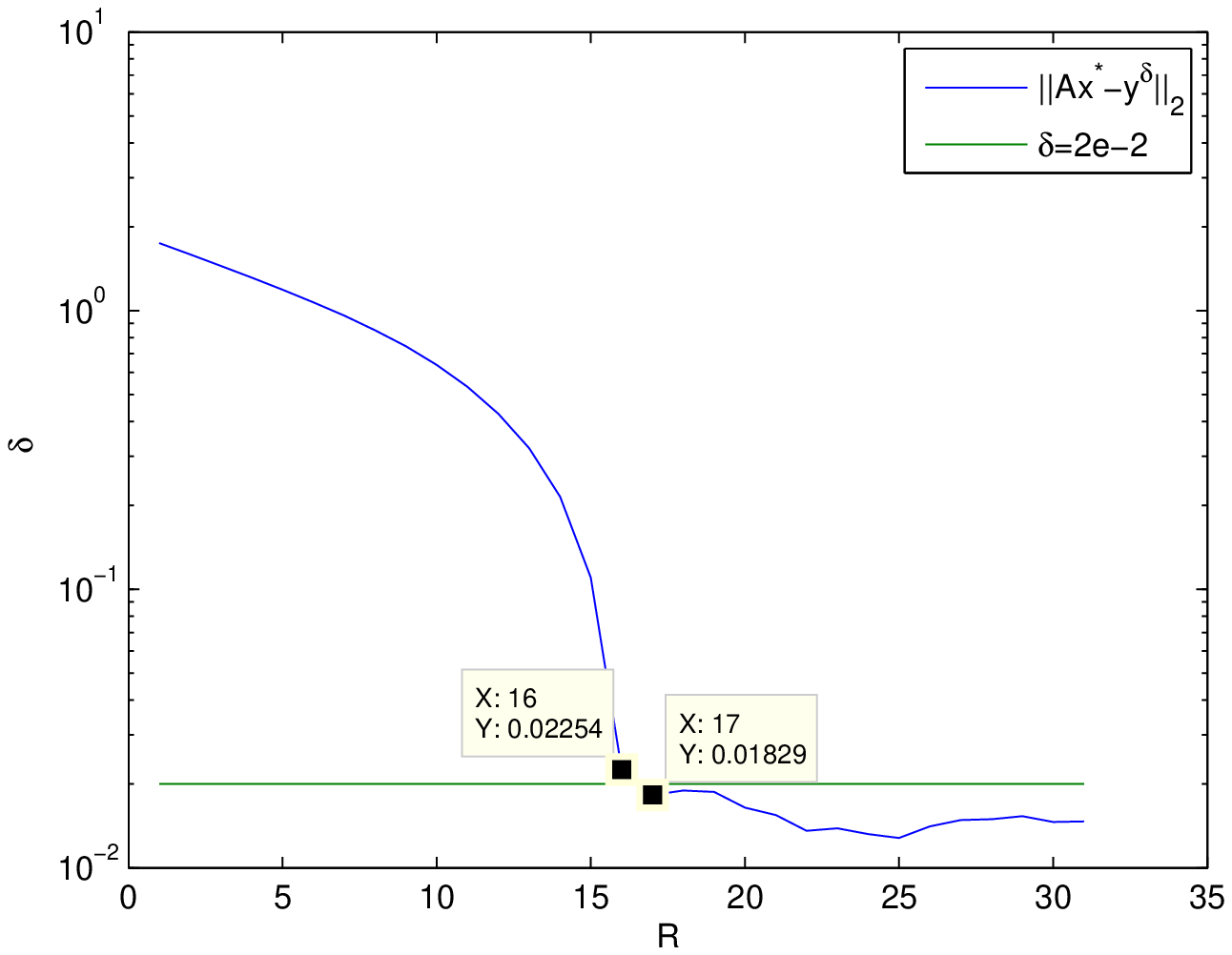}}
\subfigure[PG-SF algorithm]{\includegraphics[width=80mm,height=60mm]{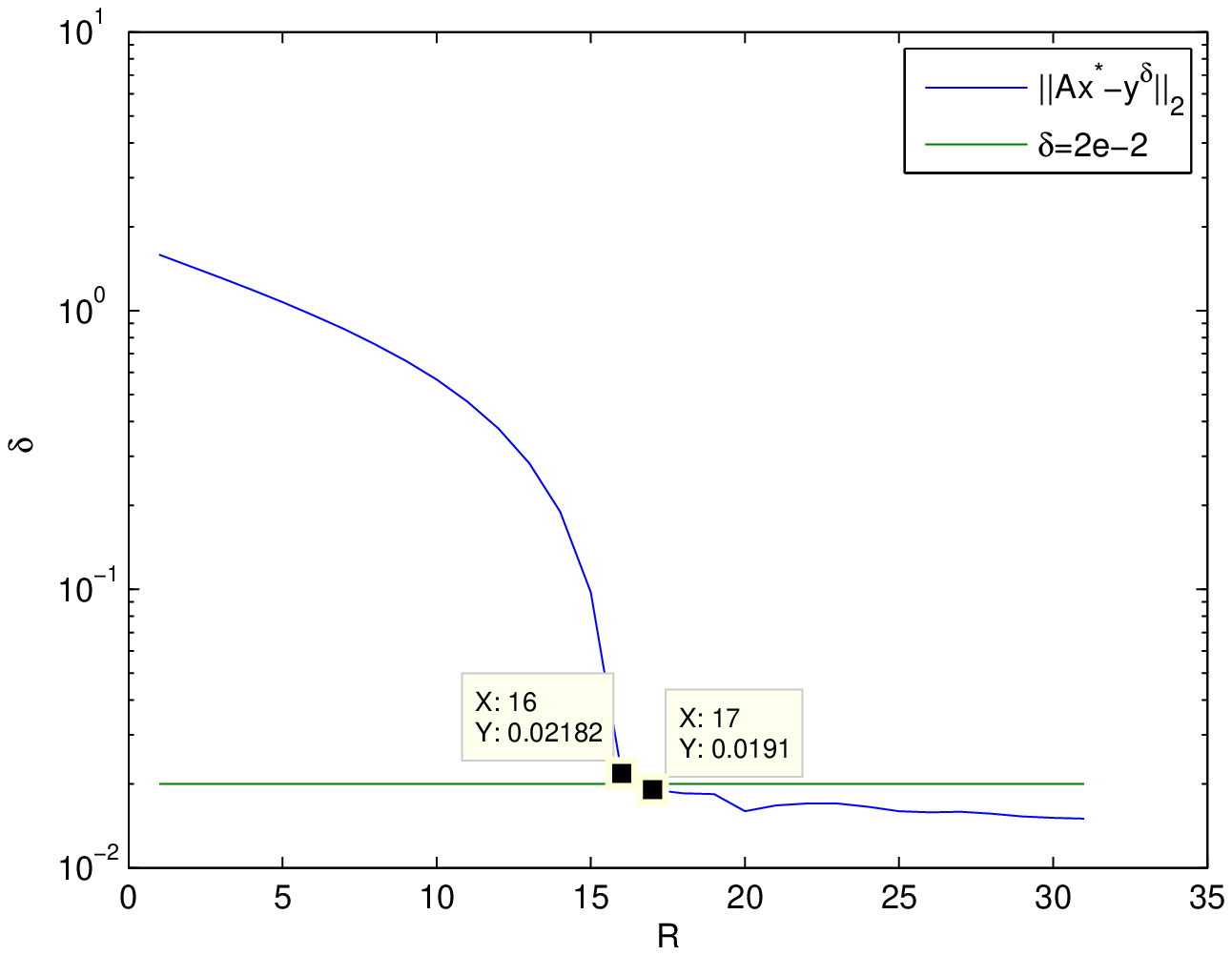}}
\caption{The discrepancy $\|Ax^*-y^{\delta}\|_2$ vs.\ $R$.}
\label{fig1}
\end{figure}

To test the stability of the PG Algorithms with respect to $R$, we choose several values of $R$ in Fig.\ \ref{fig2}. It is shown that the two PG algorithms have good performance with the appropriate radius $R$.
We see that the two PG algorithms are stable with respect to $R$. Furthermore, the results of
reconstruction get better if $R$ close to 16.

\begin{figure}[tbhp]
\centering
\subfigure[PG-GCGM algorithm]{\includegraphics[width=80mm,height=60mm]{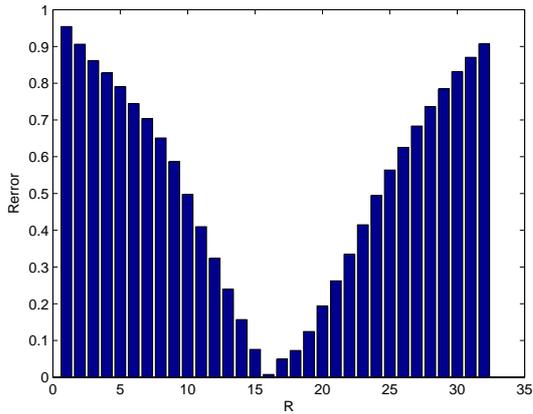}}
\subfigure[PG-SF algorithm]{\includegraphics[width=80mm,height=60mm]{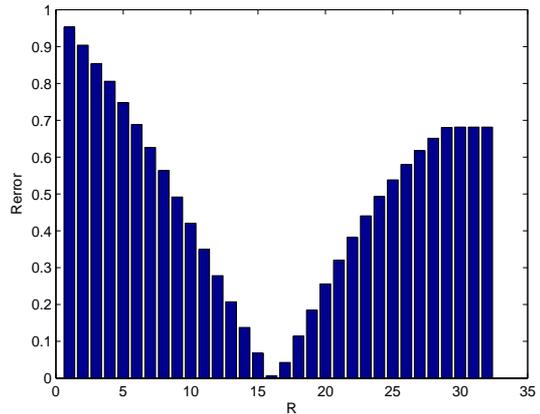}}
\caption{The relative error of reconstruction $x^*$ by the two PG algorithms with different $R$.}
\label{fig2}
\end{figure}

When $0<\eta\leq 1$, $\mathcal{R}_{\eta}(x)$ is non-convex. To analyze the influence of $\eta$, we choose different values for the parameter $\eta$.
From each row in Table \ref{tab1}, we see that, Rerror of
reconstruction gets better with $\eta$ increasing which implies the non-convex regularization (case $\eta>0$)
has better performance compared to the classical $\ell_1$ regularization (case $\eta=0$).

\begin{table}[tbhp]
\caption{Rerror of reconstruction $x^{*}$ with different values of $\eta$.}
\label{tab1}
\begin{tabular}{cccccccccc}\hline
$\displaystyle \eta$& 0.0& 0.1& 0.2& 0.3& 0.4& 0.5& 0.7& 0.9&1.0\\
\hline
ST-($\alpha\ell_{1}-\beta\ell_{2}$)&0.0250& 0.0246& 0.0147&0.0098&0.0086&0.0081&0.0073&0.0067&0.0064\\
PG-GCGM&0.0180& 0.0126& 0.0102&0.0089&0.0081&0.0074&0.0067&0.0061&0.0059\\
PG-SF&0.0356& 0.0285& 0.0197&0.0145&0.0121&0.0111&0.0096&0.0091&0.0089\\
\hline
\end{tabular}
\end{table}

We test the convergence rate of the two PG algorithms and the ST-($\alpha\ell_{1}-\beta\ell_{2}$) algorithm. We are primarily interested in the time of computation corresponding to Rerror. The results are shown in Fig.\ \ref{fig3}. To get within a distance of
the true minimizer corresponding to a 7e-3 relative error, PG-GCGM algorithm takes 0.62 second, PG-SF algorithm 1.08 seconds, and ST-($\alpha\ell_{1}-\beta\ell_{2}$) algorithm 18.40 seconds. The ST-($\alpha\ell_{1}-\beta\ell_{2}$) algorithm procedure is significantly slower than the two PG algorithms.

\begin{figure}[tbhp]
\centering
\subfigure[PG-GCGM algorithm and PG-SF algorithm]{\includegraphics[width=80mm,height=60mm]{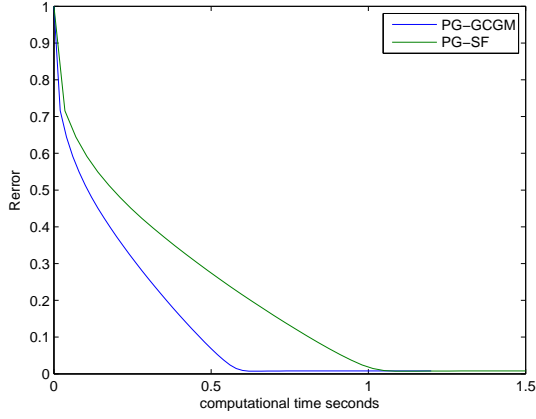}}
\subfigure[ST-($\alpha\ell_{1}-\beta\ell_{2}$) algorithm]{\includegraphics[width=80mm,height=60mm]{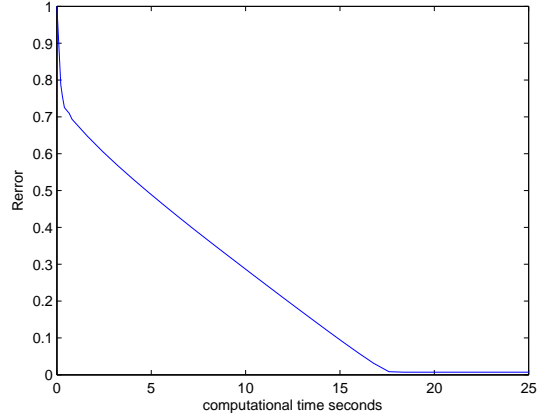}}
\caption{(a) Convergence rate of PG-GCGM algorithm and PG-SF algorithm; (b) Convergence rate of ST-($\alpha\ell_{1}-\beta\ell_{2}$) algorithm.}
\label{fig3}
\end{figure}

Theoretically, for the PG-SF algorithm, we require that Assumption \ref{assumption3_8} (A2) holds, i.e.\ $\lambda\geq \beta\max\{ {\rm eig}(x^k)\}$. Next, we test whether $\lambda$ satisfies this assumption. Fig.\ \ref{fig4}\,(a) shows Rerror corresponding to the different reconstruction $x^k$, $1\le k\le 1500$ and Fig.\ \ref{fig4}\,(b) shows the maximal eigenvalues $\max\{{\rm eig} (x^k)\}$. It is obvious that all $\max\{{\rm eig}(x^k)\}$ are less than 3.5. In this section, we let $\lambda=1$ and $\beta=\alpha\eta$, where $\alpha=0.02$ and $\eta=1$. Thus, $\lambda\geq 3.5\beta$, which satisfies Assumption \ref{assumption3_8} (A2). Theoretically, we can let $\lambda$ be any value greater than $3.5\beta$. Nevertheless, a larger value of $\lambda$ corresponds to a smaller iteration step, and then we can not obtain a good convergence rate.

\begin{figure}[tbhp]
\centering
\subfigure[]{\includegraphics[width=80mm,height=60mm]{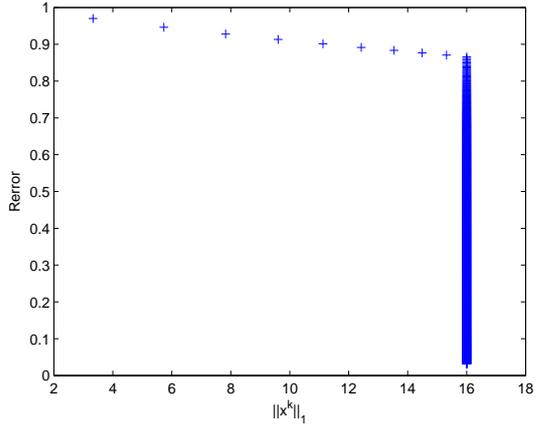}}
\subfigure[]{\includegraphics[width=80mm,height=60mm]{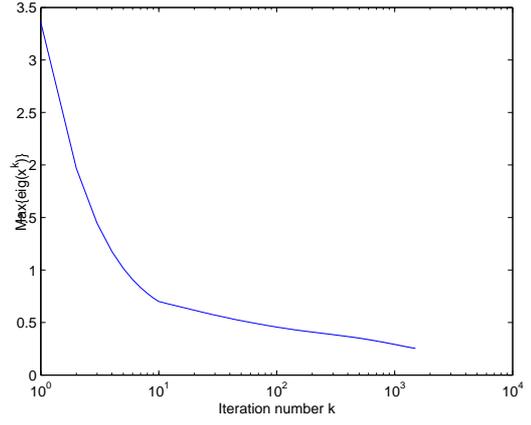}}
\caption{(a) Rerror for $x^k$, $1\le k\le 1500$; (b) $\max\{{\rm eig}(x^k)\}$ for $1\le k\le 1500$.}
\label{fig4}
\end{figure}

Finally, we let $n=1800$, $m=0.4n$ and $s=0.2m$. $\sigma=50 \rm dB$. The coefficients $\lambda$ and $\eta$ remain the same as in the first test. The noise level $\delta$ is around 0.09, hence we let $\beta=0.1$. We test the convergence rate of the two PG algorithms and ST-($\alpha\ell_{1}-\beta\ell_{2}$) algorithm regarding computational time with several different values of Rerror. With the value of Rerror decreasing, when Rerror gets within each value, we check the computational time of the three algorithms.
In Table \ref{tab2}, we see that the ST-($\alpha\ell_{1}-\beta\ell_{2}$) algorithm takes more than 100 minutes to get within a distance of
the true minimizer corresponding to a 2\% relative error. The two PG algorithms only take around 8 and 41 seconds to reach the same
level of relative error. The PG algorithms converge much faster than the ST-($\alpha\ell_{1}-\beta\ell_{2}$) algorithm.

\begin{table}[tbhp]
\caption{Time of computation for the reconstruction $x^{*}$ with different values of Rerror.}
\label{tab2}
\begin{tabular}{cccc}\hline
Rerror & ST-($\alpha\ell_{1}-\beta\ell_{2}$) time&PG-GCGM time& PG-SF time\\
\hline
0.8&9.7463 m& 0.0214 s& 0.0208 s\\
0.6&12.7113 m& 0.1926 s& 0.8573 s\\
0.4&14.9283 m& 0.6995 s& 3.2097 s\\
0.2&24.8903 m& 1.6924 s& 7.5099 s\\
0.1&39.2569 m& 2.8578 s& 11.1562 s\\
0.05&60.5784 m& 4.9201 s& 22.2830 s\\
0.02&102.8623 m& 8.2870 s& 41.2480 s\\
\hline
\end{tabular}
\end{table}

\subsection{Example 2: Image deblurring}

In the second example, we test an ill-conditioned image deblurring problem which is the process of removing blurring artifacts from images, such as blur caused by defocus aberration or motion blur. The blur is typically modeled by a Fredholm integral equation of the first kind
\[ \int_a^b K(s,t)\,f(t)\,dt=g(s),\]
where $K(s,t)$ is the kernel function, $g(s)$ is the observed image and $f(t)$ is the true image. We utilize the blur problem from MATLAB Regularization Tools
(\cite{H07}) by calling $[A,b,x^{\dag}]=\mathrm{blur}(n,band,\tau)$, where the Gaussian point-spread function is used as the kernel function
\[ K(s,t)=\frac{1}{\pi \tau^2}\mathrm{exp}\left(-\frac{s^2+t^2}{2\tau^2}\right).\]
The matrix $A$ is a symmetric $n^2\times n^2$ Toeplitz matrix and is given by $A=(2\pi \tau^2)^{-1}T\otimes T$,
where $T$ is an $n\times n$ symmetric banded Toeplitz
matrix whose first row is obtained by calling
\[z=[\mathrm{exp}(-([0:band-1].{\rm\hat{\phantom{a}}}2)/(2\tau{\rm\hat{\phantom{a}}}2));\mathrm{zeros}(1,N-band)].\]
The parameter $\tau$ controls the shape of the Gaussian point spread function and thus
the amount of smoothing (the larger the value of $\tau$, the wider the function, and
the less ill-posed the problem).

We choose $n=64$, $band=3$, $\tau=0.7$. A noise $\delta$ is added to exact data $y^{\dag}$ by calling
$\mathrm{y^{\delta}=awgn(Ax^{\dag}, \sigma})$, where $\mathrm{\sigma}=50\mathrm{dB}$, $\delta$ is around 0.2. We let $\lambda = 5$, $\eta=0.7$, $\alpha=O(\delta)=0.2$, $\beta=\alpha\eta=0.14$ and generate the initial vector $x^0$ by calling $x^0=0.01\mathrm{ones}(n,1)$. The value of $\|A\|_2$ is
around 1 and the condition number is around 30.
The initial value $x^0$ is generated by calling $x^0=0.01\mathrm{ones}(n\times n,1)$.
Fig.\ \ref{fig5} shows Morozov's discrepancy principle for determining the radius $R$. We see that the value of the discrepancy $\|Ax^*-y^{\delta}\|_2$ decreases with increasing radius $R$. According to the strategy stated previously, $R$ should be chosen such that $R=\sup\{R>0 \mid \delta< \|Ax^*-y^{\delta}\|_2\}$. It is obvious that $R$ should be chosen as 2107. Actually, the optimal $R$ is 2108 (see Fig.\ \ref{fig6}), thus the results of the experiment testify the theory, i.e.\ $R$ should be chosen by $R=\|x^*\|_1$. Note that $\|x^{\dag}\|_1=2111$. Fig.\ 6 shows the performance of the PG algorithms with respect to $R$. It is shown that the two PG algorithms have good performance with appropriate radius $R$.
Observe that for a fixed parameter $\eta$, Rerror of reconstruction $x^*$ gets better if $R$ close to 2107.

\begin{figure}[tbhp]
\centering
\subfigure[PG-GCGM algorithm]{\includegraphics[width=80mm,height=60mm]{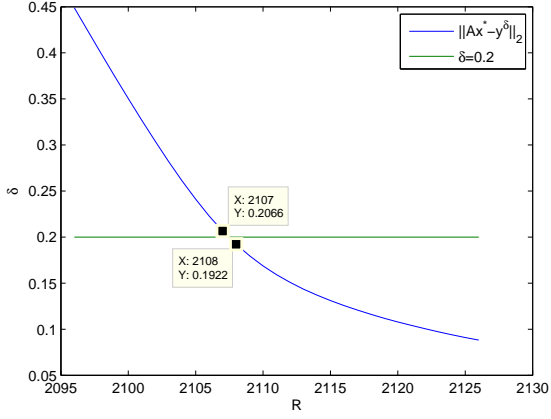}}
\subfigure[PG-SF algorithm]{\includegraphics[width=80mm,height=60mm]{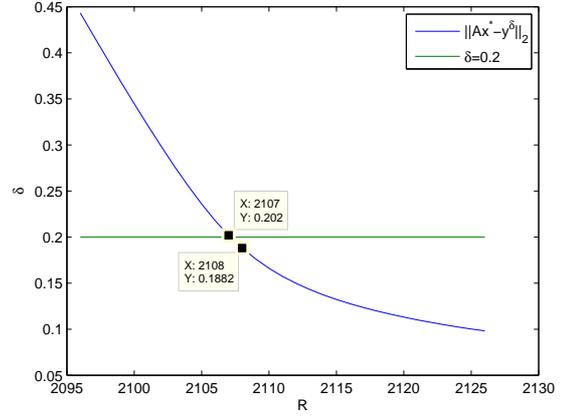}}
\caption{The value of the discrepancy $\|Ax^*-y^{\delta}\|_2$ with different $R$.}
\label{fig5}
\end{figure}

\begin{figure}[tbhp]
\centering
\subfigure[PG-GCGM algorithm]{\includegraphics[width=80mm,height=60mm]{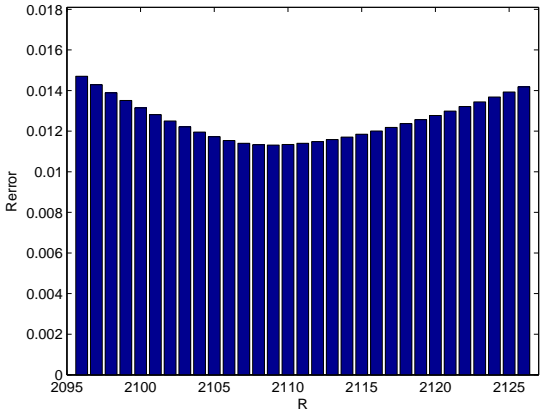}}
\subfigure[PG-SF algorithm]{\includegraphics[width=80mm,height=60mm]{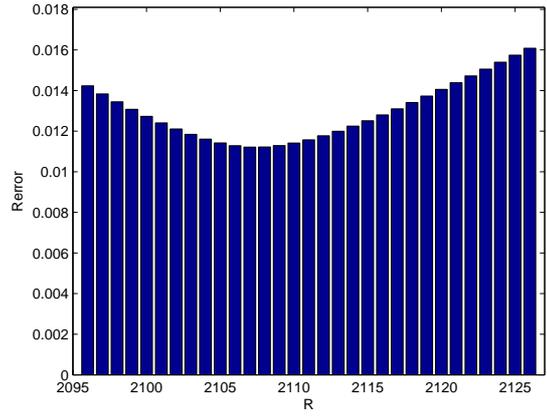}}
\caption{The relative error of reconstruction $x^*$ by the two PG algorithms with different $R$.}
\label{fig6}
\end{figure}

To analyze the influence of $\eta$, we choose different values for the parameter $\eta$.
From each row in Table \ref{tab3}, we see that the results of
reconstruction get better with $\eta$ increasing, implying that the non-convex regularization (for $\eta>0$)
has better performance than the classical $\ell_1$ regularization (for $\eta=0$). However, if $\eta$
increases to near 1, the accuracy of recovery decreases and $\eta=0.7$ is optimal.

\begin{table}[tbhp]
\caption{Rerror of reconstruction $x^{*}$ with different values of $\eta$.}
\label{tab3}
\begin{tabular}{cccccccccc}\hline
$\displaystyle \eta$& 0.0& 0.1& 0.2& 0.3& 0.4& 0.5& 0.7& 0.9&1.0\\
\hline
ST-($\alpha\ell_{1}-\beta\ell_{2}$)&0.0265& 0.0253& 0.0231&0.0205&0.0163&0.0144&0.0125&0.0138&0.0198\\
PG-GCGM&0.0278& 0.0263& 0.0242&0.0225&0.0198&0.0162&0.0130&0.0152&0.0205\\
PG-SF&0.0296& 0.0271& 0.0237&0.0231&0.0204&0.0156&0.0126&0.0147&0.0203\\
\hline
\end{tabular}
\end{table}

We test the convergence rate of the two PG algorithms and the ST-($\alpha\ell_{1}-\beta\ell_{2}$) algorithm, focusing on the
computation time corresponding to Rerror. The results are shown in Fig.\ \ref{fig7}. To get within a distance of
the true minimizer corresponding to a 1.2e-2 relative error, the PG-GCGM algorithm takes 10.12 seconds, PG-SF algorithm 36.26 seconds,
and the ST-($\alpha\ell_{1}-\beta\ell_{2}$) algorithm 58.54 minutes. The ST-($\alpha\ell_{1}-\beta\ell_{2}$) algorithm procedure is significantly slower than the two PG algorithms.

\begin{figure}[tbhp]
\centering
\subfigure[PG-GCGM algorithm and PG-SF algorithm]{\includegraphics[width=80mm,height=60mm]{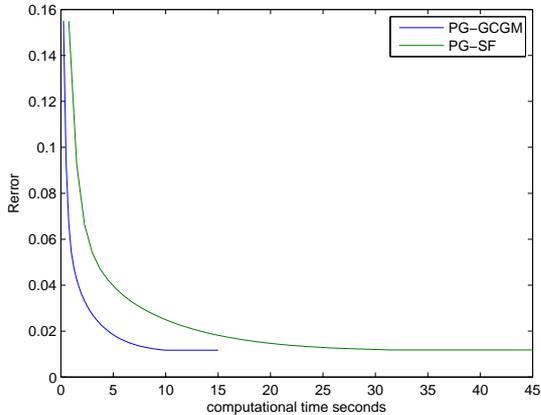}}
\subfigure[ST-($\alpha\ell_{1}-\beta\ell_{2}$) algorithm]{\includegraphics[width=80mm,height=60mm]{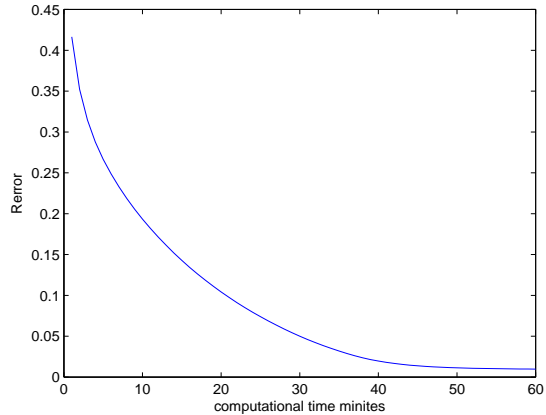}}
\caption{(a) Convergence rate of PG-GCGM algorithm and PG-SF algorithm; (b) Convergence rate of ST-($\alpha\ell_{1}-\beta\ell_{2}$) algorithm.}
\label{fig7}
\end{figure}

Theoretically, for the PG-SF algorithm, we require that Assumption \ref{assumption3_8} (A2) holds, i.e.\ $\lambda\geq \beta\max\{ {\rm eig}(x^k)\}$. In Fig.\ \ref{fig8}, we test whether $\lambda$ satisfies this assumption. Fig.\ \ref{fig8}\,(a) shows Rerror corresponding to the different reconstruction $x^k$ and Fig.\ \ref{fig8}\,(b) shows the maximal eigenvalue $\max\{{\rm eig} (x^k)\}$. It is obvious that the maximal eigenvalue of all $x^k$ is less than 0.45. We let $\lambda=1$ and $\beta=\alpha\eta=0.14$, where $\alpha=0.2$ and $\eta=0.7$. Thus, $\lambda\geq 3.5\beta$, and Assumption \ref{assumption3_8} (A2) is satisfied.

\begin{figure}[tbhp]
\centering
\subfigure[]{\includegraphics[width=80mm,height=60mm]{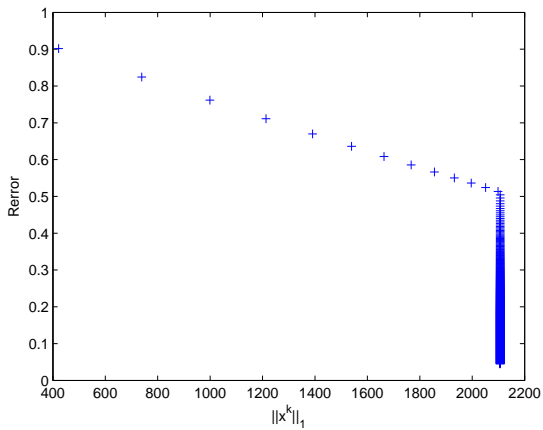}}
\subfigure[]{\includegraphics[width=80mm,height=60mm]{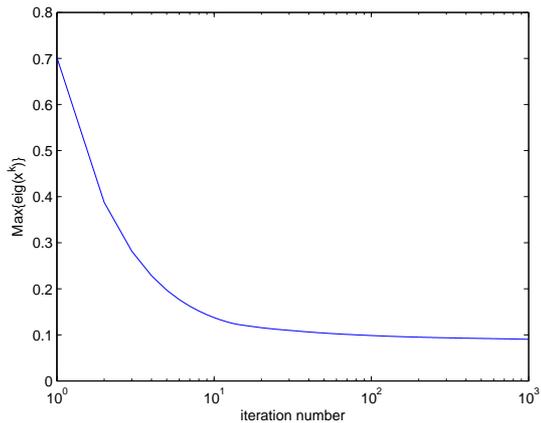}}
\caption{(a) Rerror for $x^k$, $1\le k\le 1000$; (b) $\max\{{\rm eig}(x^k)\}$ for $1\le k\le 1000$.}
\label{fig8}
\end{figure}


\end{document}